\documentclass[10pt,a4paper,twoside]{amsart}
\usepackage{hyperref}
\usepackage[english]{babel}
\usepackage{enumerate,colonequals,expdlist}
\usepackage{amssymb,amsmath,amsfonts,amsthm}
\usepackage{color,graphicx}
\usepackage{tikz,float}
\usetikzlibrary{patterns}
\usepackage[textsize=tiny]{todonotes} 
\usepackage{mathabx}
\usepackage[noadjust]{cite}

\theoremstyle{plain}
\newtheorem{thm}{Theorem}[section]
\newtheorem{prop}[thm]{Proposition}
\newtheorem{cor}[thm]{Corollary}
\newtheorem{lemma}[thm]{Lemma}
\theoremstyle{definition}

\newtheorem{ex}[thm]{Example}
\theoremstyle{remark}
\newtheorem{rmk}[thm]{Remark}
\numberwithin{equation}{section}
\numberwithin{figure}{section}

\newcommand{\abs}[1]{\lvert{#1}\rvert}    
\newcommand{\normsymb}{\|}
\newcommand{\norm}[2]{\normsymb{#1}\normsymb_{#2}}  

\newcommand{\RR}{\mathbb{R}} 
\newcommand{\CC}{\mathbb{C}} 
\newcommand{\NN}{\mathbb{N}} 
\newcommand{\ZZ}{\mathbb{Z}} 

\newcommand{\cE}{{\mathcal E}}
\newcommand{\cH}{{\mathcal H}}
\newcommand{\cM}{{\mathcal M}}

\newcommand{\cR}{{\mathcal R}}
\newcommand{\cT}{{\mathcal T}}

\newcommand{\fh}{{\mathfrak h}}
\newcommand{\fv}{{\mathfrak v}}

\DeclareMathOperator{\dd}{d\!}  
\newcommand{\ee}{\mathrm e}     
\newcommand{\EE}{\mathsf{E}}    
\DeclareMathOperator{\Dom}{{\mathcal D}} 
\DeclareMathOperator{\Ran}{\mathrm{Ran}} 
\renewcommand{\epsilon}{\varepsilon}
\DeclareMathOperator{\supp}{supp}

\DeclareMathOperator{\diam}{diam}
\DeclareMathOperator{\diag}{diag}
\DeclareMathOperator{\Span}{span}

\newcommand{\hm}[1]{\leavevmode{\marginpar{\tiny%
 $ \hbox to 0mm{\hspace*{-0.5mm} $ \leftarrow $ \hss}%
 \vcenter{\vrule depth 0.1mm height 0.1mm width \the\marginparwidth}%
 \hbox to
 0mm{\hss $ \rightarrow $ \hspace*{-0.5mm}} $ \\\relax\raggedright #1}}}

\newcommand{\bijection}{\Psi}
 
\begin{document}
\title[An abstract Logvinenko-Sereda type theorem]{An abstract Logvinenko-Sereda type theorem for spectral subspaces}
\author[M.~Egidi]{Michela Egidi}
\address[M.~Egidi]{Ruhr Universit\"at Bochum, Fakult\"at f\"ur Mathematik, D-44780 Bochum, Germany}
\email{michela.egidi@ruhr-uni-bochum.de}

\author[A.~Seelmann]{Albrecht Seelmann}
\address[A.~Seelmann]{Technische Universit\"at Dortmund, Fakult\"at f\"ur Mathematik, D-44221 Dortmund, Germany}
\email{albrecht.seelmann@math.tu-dortmund.de}

\thanks{\copyright 2021. This manuscript version is made available under the CC-BY-NC-ND 4.0 license
\texttt{http://creativecommons.org/licenses/by-nc-nd/4.0/}. The final authenticated version is available online at
\texttt{https://doi.org/10.1016/j.jmaa.2021.125149}}

\begin{abstract}
	We provide an abstract framework for a Logvinenko-Sereda type theorem, where the classical compactness assumption on the support
	of the Fourier transform is replaced by the assumption that the functions under consideration belong to a spectral subspace
	associated with a finite energy interval for	some lower semibounded self-adjoint operator on a Euclidean $L^2$-space. Our result
	then provides a bound for the $L^2$-norm of such functions in terms of their $L^2$-norm on a thick subset with a constant
	explicit in the geometric and spectral parameters.	This recovers previous results for functions on the whole space,
	hyperrectangles, and infinite strips with compact Fourier support and for finite linear combinations of Hermite functions and
	allows to extend them to other domains.
	
	The proof follows the approach by Kovrijkine and is based on Bernstein-type inequalities for the respective functions,
	complemented with a suitable covering of the underlying domain.
\end{abstract}

\maketitle

\section{Introduction and main results}

The classical uncertainty principle of Heisenberg states that it is impossible to precisely determine the position and the
momentum of a particle at the same time. A typical manifestation of this principle is the fact that a function and its Fourier
transform cannot be localized simultaneously. Over the years, this has been expressed in different ways and has taken different
mathematical formulations. For an exhaustive presentation, we refer the reader to the monograph~\cite{havin-joericke:94}. 

One of these formulations is connected to the definition of so-called~\emph{strongly annihilating pairs}:
For $\mathcal{S}, \mathcal{E} \subset\RR^d$, the pair ($\mathcal{S}, \mathcal{E})$ is called strongly annihilating if there
exists a constant $C=C(\mathcal{S}, \mathcal{E})$ such that 
\begin{equation}\label{eq:annihilating}
 \norm{f}{L^2(\RR^d)}^2
 \leq
 C (\norm{f}{L^2(\RR^d\setminus\mathcal{S})}^2 + \norm{\hat{f}}{L^2(\RR^d\setminus\mathcal{E})}^2)
 \qquad \forall\, f\in L^2(\RR^d)
 ,
\end{equation}
where $\hat{\cdot}$ denotes the Fourier transform.

The Amrein-Berthier theorem~\cite{amrein-berthier:77} states that if $\mathcal{S}$ and $\mathcal{E}$ both have finite Lebesgue
measure, then the pair $(\mathcal{S},\mathcal{E})$ is strongly annihilating. Nazarov~\cite{nazarov:92} provided in dimension
$d = 1$ the explicit constant $C = c \ee^{c\abs{\mathcal{S}}\abs{\mathcal{E}}}$, which is optimal up to the universal constant
$c>0$. Later, Jaming~\cite{jaming:07} extended this result to higher dimensions with
$C = c_d \ee^{c_d\min\{\abs{\mathcal{S}}\abs{\mathcal{E}},\abs{\mathcal{S}}^{1/d}w(\mathcal{E}),
w(\mathcal{S})\abs{\mathcal{E}}^{1/d}\}}$, where $w(\cdot)$ denotes the mean width of the respective set.
A complete classification of all strongly annihilating pairs seems currently out of reach. However, it is known that the pair
$(\mathcal{S},\mathcal{E})$ is strongly annihilating if and only if
\begin{equation}\label{eq:spectral-fourier}
 \norm{f}{L^2(\RR^d)}^2
 \leq
 \tilde {C}(\mathcal{S}, \mathcal{E}) \norm{f}{L^2(\mathcal{S})}
 \qquad \forall \, f\in L^2(\RR^d) \text{ with }\supp \hat f\subset \mathcal{E},
\end{equation}
see, e.g.,~\cite[Section~3.\S{}1.1]{havin-joericke:94}, which is another and more popular form of an uncertainty relation. Under
the additional assumption that $\mathcal{E}$ is bounded and in the context of finding all sets $\mathcal{S}$ for which the
seminorm $\norm{\cdot}{L^2(\mathcal{S})}$ is a norm equivalent to $\norm{\cdot}{L^2(\RR^d)}$, this has been studied by
Panejah~\cite{panejah:61,panejah:62}, Kacnel'son~\cite{kacnelson:73}, and Logvinenko-Sereda~\cite{Logvinenko-Sereda-74}, whose
results can be summarised as follows.

\begin{thm}[Panejah, Kacnel'son, Logvinenko-Sereda]\label{thm:PKLS}
	Let $\mathcal{S}\subset\RR^d$ be measurable and $r>0$. Then, the following statements are equivalent: 
	\begin{enumerate}[(i)]
		\item
		There exists a constant $C=C(\mathcal{S}, r)>0$ such that inequality~\eqref{eq:spectral-fourier} holds for all	functions
		$f\in L^2(\RR^d)$ with $\supp\hat f\subset B(x, r)$, $x\in\RR^d$;

		\item
		The set $\mathcal{S}$ is thick, that is, there exist parameters $\gamma\in(0,1]$ and $\rho > 0$ such that
		$\abs{\mathcal{S}\cap (x + (0,\rho)^d)} \geq \gamma\rho^d$ for all $x\in\RR^d$.
	\end{enumerate}
\end{thm}
Kacnel'son and Logvinenko-Sereda also provide the constant $C= c_1e^{c_2 r}$ with $c_1, c_2 > 0$ depending only on $\gamma$ and
$\rho$. However, the optimal behaviour has been discovered much later by Kovrijkine~\cite{Kovrijkine:01, Kovrijkine:thesis}, who
gives the explicit form
\[
 C=\Big(\frac{K^d}{\gamma}\Big)^{Krd\rho + d}
\]
with a universal constant $K>0$.

It is worth to note that Theorem~\ref{thm:PKLS} was originally formulated for functions in $L^p(\RR^d)$, $p\in (0,\infty)$, and
that the work of Kovrijkine is valid for all $p\in [1,\infty]$ and has nowadays found a broad application in the realm of control
theory, see, e.g.,~\cite{EV18,martin-pravdastarov:20,GST20,BGST20}.

Recently, statements analogous to Theorem~\ref{thm:PKLS} for $L^2$-functions defined on a hyperrectangle with finite Fourier
series, and for $L^2$-functions on an infinite strip with finite Fourier series in the bounded coordinates and compactly
supported Fourier transform in the unbounded coordinates, have been proved in~\cite{EV20,Egi21}, respectively, following the
ideas of Kovrijkine in order to obtain the best possible constant.

If one aims for related statements in other situations, it is natural to replace the assumption of compact Fourier support by
other restrictions. Following Kovrijkine's approach, this has been done in terms of the Fourier-Bessel (or Hankel) transform
in~\cite{ghobber-jaming:13} and in terms of model spaces in~\cite{HJK20}.
In the present work, we consider in the Euclidean
setting functions belonging to a spectral subspace for some self-adjoint operator associated with a finite energy interval. More
precisely, given some domain $\Omega\subset\RR^d$, a lower semibounded self-adjoint operator $H$ on $L^2(\Omega)$, a measurable
subset $\omega\subset\Omega$,
and $\lambda \in \RR$, one asks whether there exists a constant $C>0$ such that 
\begin{equation}\label{eq:spectral-abstract}
 \norm{f}{L^2(\Omega)}^2
 \leq
 C \norm{f}{L^2(\omega)}^2
 \quad \forall f\in \Ran\EE_{H}(\lambda)
 ,
\end{equation}
where $\EE_{H}(\lambda)$ denotes the spectral projection for $H$ up to energy value $\lambda$ and the constant $C$ depends on
$H, \Omega, \omega$ and $\lambda$.
An inequality of the above form is usually called a~\emph{spectral inequality} in the context of control theory
and~\emph{unique continuation principle} in the theory of (random) Schr\"odinger operators. We refer the reader
to~\cite{ENSTTV20,NakicTTV-COVC} and~\cite{Klein13,RMV13,NTTV18,ST20}, respectively, and the works cited therein for an overview
of the different uses of~\eqref{eq:spectral-abstract} in both these fields.

In order to be useful in the context of control theory, the constant $C$ in~\eqref{eq:spectral-abstract} should have the form
\begin{equation}\label{eq:spectralIneqControl}
 C = d_1 \ee^{d_2\max\{\lambda,0\}^s},
 \quad s\in(0,1).
\end{equation}
If $H$ is the Dirichlet Laplacian on a bounded domain and $\omega$ is chosen to be an open subset, this has been established with
$s=1/2$ by Lebeau and Robbiano in~\cite{lebeau-robbiano-95}. Under additional assumptions on $\Omega$, it is also possible to
choose $\omega$ to be just measurable with positive measure as the work~\cite{ApraizEWZ-14} shows. However, if one wants to
discuss certain asymptotic regimes as was done in~\cite{NakicTTV-COVC}, one desires a more explicit knowledge on the dependence
of $d_1$ and $d_2$ on certain geometric parameters of $\omega$. If $\omega$ is a so-called~\emph{equidistributed set}, that is,
$\omega$ contains a union of suitably distributed open balls of fixed radius, this was achieved for Schr\"odinger operators with
bounded potentials on hyperrectangles in~\cite{NTTV18}, and later also on arbitrary rectangular shaped domains
in~\cite{NakicTTV-JST}, with the use of Carleman estimates; for earlier works in this direction see also the references cited
in both these papers. A related quantitative result for certain unbounded potentials can be found in~\cite{KT16}.
However, if $\omega$ is just assumed to be a thick set, so far other techniques based on complex analysis are required. In this
regard, Schr\"odinger operators on $\RR^d$ with certain analytic potentials have been considered in~\cite{lebeau-moyano-19},
which is closely related to the work by Kacnel'son but does not provide an explicit form for the constants $d_1$ and $d_2$. Such
explicit results for thick sets are, as of now, known only for results based on Kovrijkine's approach. Here, his works handle the
case of the pure Laplacian on $\RR^d$, where thick sets describe the complete class of sets for which~\eqref{eq:spectral-abstract}
holds in this situation. Following Kovrijkine's ideas, the above mentioned works~\cite{EV20} and~\cite{Egi21}
establish~\eqref{eq:spectral-abstract} for the pure Laplacian on hyperrectangles and infinite strips, respectively. Moreover,
based on essentially the same reasoning, finite linear combinations of the standard Hermite functions have been treated
in~\cite{BJPS18}, establishing~\eqref{eq:spectral-abstract} also for the harmonic oscillator on $\RR^d$. The question arising now
is how much one can stretch these techniques to study~\eqref{eq:spectral-abstract} with thick sets $\omega$ on other domains
and/or for other operators.

\subsection{A Logvinenko-Sereda-type theorem}
The main aim of the present paper is to establish a general framework that covers the above mentioned results
from~\cite{Kovrijkine:01,Kovrijkine:thesis,EV20, Egi21, BJPS18} and, at the same time, allows to extend them to a larger class of
domains and operators. In this regard, we consider measurable subsets $\omega$ of a domain $\Omega\subset\RR^d$ that
are~\emph{thick in $\Omega$}, that is, there exist $\gamma\in(0,1]$ and $\rho>0$ such that
\begin{equation*}
 \abs{ \omega \cap (x + (0,\rho)^d) } \ge \gamma \rho^d
\end{equation*}
for all $x \in \RR^d$ with $x + (0,\rho)^d \subset \Omega$, where $\abs{\,\cdot\,}$ denotes the Lebesgue measure. In this case,
we also refer to $\omega$ as $(\gamma,\rho)$-thick in $\Omega$ to emphasise the parameters. A particular instance of these are
sets of the form $\omega = \Omega \cap S$ with some set $S \subset \RR^d$ that is $(\gamma,\rho)$-thick in $\RR^d$, in which case
the thickness parameters do not depend on $\Omega$.

The material we present here builds upon the mentioned works and draws on two central concepts. The first one is the concept
of~\emph{Bernstein-type inequalities}:
We suppose that for all $\lambda \in \RR$ we have
$\Ran \EE_H(\lambda) \subset W^{\infty,2}(\Omega) = \bigcap_{k\in\NN} W^{k,2}(\Omega)$, where $W^{k,2}(\Omega)$ is the
$L^2(\Omega)$-Sobolev space of order $k$, and that for some function $C_B \colon \NN_0 \times \RR \to (0,\infty)$ every
$f \in \Ran\EE_H(\lambda)$ satisfies
\begin{equation}\label{eq:Bernstein}
 \sum_{\abs{\alpha}=m} \frac{1}{\alpha!} \norm{\partial^\alpha f}{L^2(\Omega)}^2
 \le
 \frac{C_B(m,\lambda)}{m!} \norm{f}{L^2(\Omega)}^2
\end{equation}
for all $m \in \NN_0$; of course, for $\lambda$ below the infimum of the spectrum of $H$, the choice of $C_B(\cdot,\lambda)$ is
completely arbitrary. These estimates resemble (for $m=1$) the well-known Bernstein inequality from Fourier analysis, where the
$L^2$-norm of the partial derivatives of a function with compactly supported Fourier transform can be bounded by its $L^2$-norm
times a constant depending on the support of the Fourier transform, see, e.g.,~\cite[Theorem~11.3.3]{Boas54}
or~\cite[Proposition~1.11 and Lemma~10.1]{MS13}.

The second concept we use is that of a suitable covering of $\Omega$: Given $\kappa \ge 1$, $\varrho > 0$, $\eta >0$, and
$l = (l_1,\dots,l_d) \in [\varrho,\infty)^d$, we call a finite or countably infinite family $\{Q_j\}_{j\in J}$ of non-empty
bounded convex open subsets $Q_j \subset \Omega$ a~\emph{$(\kappa,\varrho,l,\eta)$-covering of $\Omega$} if
\begin{enumerate}[(i)]

 \item $\Omega \setminus \bigcup_{j \in J} Q_j$ has Lebesgue measure zero;\label{it:covering}

 \item $\sum_{j \in J} \norm{g}{L^2(Q_j)}^2\leq \kappa \norm{g}{L^2(\Omega)}^2$ for all $g \in L^2(\Omega)$;\label{it:overlap}

 \item each $Q_j$ contains a hypercube of the form $x + (0,\varrho)^d$, $x \in \RR^d$, and lies in a hyperrectangle of the form
       $y + \bigtimes_{k=1}^d (0,l_k)$, $y \in \RR^d$;\label{it:geometry}
       
 \item for each $Q_j$ there is a linear bijection $\bijection_j \colon \RR^d \to \RR^d$ with
       \begin{equation*}
        \abs{\bijection_j(Q_j)} \geq \eta \diam(\bijection_j(Q_j))^d.
       \end{equation*}
       \label{it:bijection}
\end{enumerate}
Here, the quantity $\kappa$ in~\eqref{it:overlap} can be interpreted as the maximal essential overlap 
between the sets $Q_j$, whereas the bijection $\bijection_j$ in~\eqref{it:bijection} may be used to 
compensate for the case where $Q_j$ has an unfavourable ratio of its volume and the $d$-th power of its diameter,
cf.~Remark~\ref{rmk:bijection} below. Of course, the existence of such a covering for given parameters $(\kappa,\varrho,l,\eta)$
depends on the geometry of $\Omega$ and requires, in particular, that at least one hypercube of the
form $x + (0,\varrho)^d$, $x \in \RR^d$, is contained in $\Omega$. 
In our context of $(\gamma,\rho)$-thick sets in $\Omega$, we usually consider $\varrho = \rho$.

We are now in position to formulate the main result of the present paper.

\begin{thm}\label{thm:LS}
 Let $\omega$ be $(\gamma,\rho)$-thick in a domain $\Omega \subset \RR^d$ with $(\kappa,\rho,l,\eta)$-covering, and let $H$ be a
 lower semibounded operator on $L^2(\Omega)$ such that for some function $C_B \colon \NN_0 \times \RR \to (0,\infty)$ every
 $f \in \Ran\EE_H(\lambda)$ satisfies a Bernstein-type inequality as in~\eqref{eq:Bernstein}.

 If for $\lambda \in \RR$ we have
 \begin{equation}\label{eq:BernsteinSum:intro}
  h(\lambda)
  :=
  \sum_{m \in \NN_0} \sqrt{C_B(m,\lambda)}\, \frac{(10\norm{l}{1})^m}{m!}
  <
  \infty
 \end{equation}
 with $\norm{l}{1} = l_1 + \dots + l_d$, then every $f \in \Ran\EE_H(\lambda)$ satisfies
 \begin{equation}\label{eq:LS}
  \norm{f}{L^2(\Omega)}^2
  \le
  \frac{\kappa}{6}
  \Bigl( \frac{24d\tau_d l_1\cdots l_d}{\gamma\eta\rho^d} \Bigr)
   ^{2\frac{\log\kappa}{\log 2} + 4\frac{\log h(\lambda)}{\log 2} + 5}
   \norm{f}{L^2(\omega)}^2
  ,
 \end{equation}
 where $\tau_d = \Gamma(d/2 +1)^{-1}\pi^{d/2}$, $\Gamma(\cdot)$ being the Gamma function, denotes the Lebesgue measure of the
 Euclidean unit ball in $\RR^d$.
\end{thm}

Following the reasoning of~\cite[Section~3.\S{}1.1]{havin-joericke:94}, Theorem~\ref{thm:LS} yields the inequality
\begin{equation*}
 \norm{f}{L^2(\Omega)}^2
 \leq
 C\bigl( \norm{f}{L^2(\omega)}^2 + \norm{(1-\EE_H(\lambda))f}{L^2(\Omega)}^2 \bigr)
 \qquad \forall\, f\in L^2(\Omega)
 ,
\end{equation*}
where the constant $C > 0$ is determined by the constant from~\eqref{eq:LS} relating $\norm{f}{L^2(\Omega)}$ and
$\norm{f}{L^2(\omega)}$ for $f \in \Ran\EE_H(\lambda)$. This can be viewed as an abstract variant of~\eqref{eq:annihilating} and
may give rise to the general notion of $(\Omega \setminus \omega, (-\infty,\lambda])$ as an annihilating pair.

Bernstein-type inequalities with condition~\eqref{eq:BernsteinSum:intro} play the central role in Theorem~\ref{thm:LS} and are
the starting point for every application. Note that condition~\eqref{eq:BernsteinSum:intro} constraints the admissible parameters
$l=(l_1,\dots,l_d)$ for the covering of $\Omega$ and, in turn, the thickness parameter $\rho$ for the set $\omega$. It also
determines the dependency of the final estimate on $\lambda$, which is of particular importance in view of the desired
form~\eqref{eq:spectralIneqControl}. Clearly, condition~\eqref{eq:BernsteinSum:intro} is governed by the constants
$C_B(m,\lambda)$, which depend on the domain $\Omega$ and the operator $H$ under consideration.

\begin{rmk}
	The factor $10$ in condition~\eqref{eq:BernsteinSum:intro} is chosen merely for	convenience and simplicity. With a suitable
	adaptation of the proof, namely Lemma~\ref{lem:kovrijkine_orig} below and the notion of good and bad elements of the covering
	in Section~\ref{subsec:good_and_bad}, this factor can be replaced by anything larger than $1$, but at the cost of a more
	involved and slightly worse constant relating $\norm{f}{L^2(\Omega)}$ and $\norm{f}{L^2(\omega)}$ in the overall
	inequality~\eqref{eq:LS}. Since this is of no importance for our applications discussed below, we refrain from pursuing this
	further here.
\end{rmk}

The rest of the paper is organised as follows. In Section~\ref{subsec:applications}, we give a brief overview on situations where
we know how to verify the hypotheses of Theorem~\ref{thm:LS} and thus obtain respective spectral inequalities of the desired
form~\eqref{eq:spectralIneqControl} with explicit control of the dependency of the involved constants on the geometric parameters
$\gamma$ and $\rho$ of the thick set. Here, the relation~\eqref{eq:Bernstein:spectral} below is of particular importance for the
case of the pure Laplacian and seems to add some novel aspects to the theory. Section~\ref{sec:bernstein} is then devoted to the
rigorous proof of the Bernstein-type inequalities for these situations, including said relation~\eqref{eq:Bernstein:spectral} for
the pure Laplacian. The proof of (a more general variant of) the main Theorem~\ref{thm:LS} is given in
Section~\ref{sec:main_proof}. Finally, some auxiliary material including the construction of suitable coverings for certain
domains is provided in the appendices.

\subsection{Applications and examples}\label{subsec:applications}
The applications we have in mind for Theorem~\ref{thm:LS} concern the situation where $H$ is given by a second order elliptic
differential operator with homogeneous Dirichlet or Neumann boundary conditions (where applicable). It turns out, however, that
if condition~\eqref{eq:BernsteinSum:intro} holds, then all functions $f \in \Ran\EE_H(\lambda)$ are necessarily analytic in
$\Omega$ in the sense that they can locally be expanded into convergent power series, see Lemma~\ref{lem:Bernstein} below. This,
in turn, is to be expected only if the coefficients of the elliptic differential operator are analytic in $\Omega$. So far, we
know how to verify~\eqref{eq:Bernstein} on certain Euclidean domains such that~\eqref{eq:BernsteinSum:intro} holds for some
$l \in (0,\infty)^d$ only for the pure Laplacian, some divergence-type operators with constant coefficient matrix, and the
harmonic oscillator, which are analysed separately below. We hope to broaden the class of operators and domains for which
Theorem~\ref{thm:LS} can be applied in future research.

It should be mentioned that some of the results discussed below could in essence be inferred also from existing spectral
inequalities by rescaling and reflection. However, our approach is based solely on manipulations on the level of the
Bernstein-type inequalities, which has the advantage that the thick set does not need to be transformed as well. Moreover, the
validity of Bernstein-type inequalities in these situations may be of independent interest.

\subsubsection{The pure Laplacian}\label{subsubsec:Laplacian}

In the case of the pure Laplacian, Bernstein-type inequalities of the form~\eqref{eq:Bernstein} most naturally follow from the
relation
\begin{equation}\label{eq:Bernstein:spectral}
	\sum_{\abs{\alpha}=m} \frac{1}{\alpha!} \norm{\partial^\alpha f}{L^2(\Omega)}^2
	=
	\frac{1}{m!} \langle f , (-\Delta_\Omega)^m f \rangle_{L^2(\Omega)}
	,\quad
	f \in \Ran\EE_{-\Delta_\Omega}(\lambda).
\end{equation}
Indeed, by functional calculus we have
$\langle f , (-\Delta_\Omega)^m f \rangle_{L^2(\Omega)} \le \lambda^m \norm{f}{L^2(\Omega)}^2$ for all
$f \in \Ran\EE_{-\Delta_\Omega}(\lambda)$, $\lambda \ge 0$, implying that in this case the constant $C_B(m, \lambda)$
in~\eqref{eq:Bernstein} is exactly $\lambda^m$ for $\lambda \ge 0$, independently of any geometric parameters of the domain
under consideration. The resulting quantity $h(\lambda)$ in~\eqref{eq:BernsteinSum:intro} then reads
\begin{equation*}
	h(\lambda) = \exp(10 \norm{l}{1} \sqrt{\lambda}),
\end{equation*}
which poses no constraints on the parameter $l$ of the covering of $\Omega$. 

In Proposition~\ref{prop:bernstein-pureLaplacian} below, we establish relation~\eqref{eq:Bernstein:spectral} for the following
domains:
\begin{itemize}  

	\item
	\emph{generalised rectangles} $\bigtimes_{j=1}^d (a_j,b_j)$ with $a_j,b_j \in \RR \cup \{\pm\infty\}$, $a_j < b_j$;
	
	\item
	equilateral triangles, isosceles right-angled triangles, and hemiequilateral triangles in $\RR^2$;
	
	\item
	sectors $S_\theta = \{ (x_1,x_2) \in (0,\infty)^2 \colon x_2 < x_1\tan \theta \} \subset \RR^2$ with
	$\theta = \pi / n$, $n \in \NN$, $n \ge 3$;
	
	\item
	any finite Cartesian product of the above.
	
\end{itemize} 
Sectors with angle $\pi/2$ and $\pi$ are actually generalised rectangles in dimension $d = 2$ and are therefore covered there.

In the case of the whole $\RR^d$, hyperrectangles, and strips of the form hypercube$\times \RR$, Bernstein-type
inequalities of the specific form~\eqref{eq:Bernstein} can easily be deduced also from the related variants used
in~\cite{Egi21,EV20,Kovrijkine:thesis,Kovrijkine:01}. The more precise relation~\eqref{eq:Bernstein:spectral}, however, is much
deeper and seems not to have been considered anywhere before. We picture this as one of the key aspects of the present work and
believe it to be of independent interest.

For each of the domains listed above, provided that it contains a hypercube of the form $x+(0,\rho)^d$, $x\in\RR^d$, it is
possible to construct a suitable covering with $\varrho = \rho$, where the other parameters $\kappa$, $l$, and $\eta$ can be
bounded efficiently in terms of $\rho$, the dimension $d$, and possibly angles $\theta$ (if sectors are involved). A
corresponding statement is given in the following lemma. Its proof is merely a technical issue and is thus deferred to
Appendix~\ref{sec:cov}.

\begin{lemma}\label{lem:covering}
	For each of the above mentioned domains containing a hypercube of the form $x + (0,\rho)^d$, $x \in \RR^d$, there is a
	$(2^d,\rho,l,(2d)^{-d/2})$-covering of $\Omega$ with $l = (k_1\rho , \dots, k_d\rho)$, where $k_1,\dots,k_d > 0$ are constants
	depending at most on the opening angles of possibly involved sectors.
\end{lemma}

Finally, using the asymptotic formula $\tau_d \sim (2\pi\ee/d)^{d/2} / \sqrt{d\pi}$, we see that $d^{1+d/2} \tau_d$ can be
bounded from above by the $d$-th power of some absolute constant. Consequently, putting all of the above considerations together,
Theorem~\ref{thm:LS} gives the following result.

\begin{cor}\label{cor:pureLaplacian}
	Let $\Omega \subset \RR^d$ be one of the above mentioned domains, and let $\omega$ be a $(\gamma,\rho)$-thick set in $\Omega$.
	Then, for all $f\in \Ran\EE_{-\Delta_\Omega}(\lambda)$, $\lambda \geq 0$, we have 
	\begin{equation*}
		\norm{f}{L^2(\Omega)}^2
		\leq
		\Bigl( \frac{C^d}{\gamma} \Bigr)^{C d\rho\sqrt{\lambda} + 2d + 6} \norm{f}{L^2(\omega)}^2
		,
	\end{equation*}
	where $C>0$ depends at most on the opening angles of possibly involved sectors.
\end{cor} 

If $\omega$ is chosen of the form $\omega = \Omega \cap S$ with some set $S$ that is $(\gamma,\rho)$-thick in $\RR^d$, then the
conclusion of the above corollary is uniform in all mentioned domains $\Omega$, provided that they contain a hypercube of the
form $x + (0,\rho)^d$. This is then consistent with the corresponding known results for the pure Laplacian on hyperrectangles,
strips, and the whole space $\RR^d$ proved in~\cite{EV20,Egi21,Kovrijkine:01,Kovrijkine:thesis} and, in particular, essentially
reproduces these results. Moreover, the obtained constant relating $\norm{f}{L^2(\Omega)}^2$ and $\norm{f}{L^2(\omega)}^2$ in
Corollary~\ref{cor:pureLaplacian} is optimal with respect to its dependence on the thickness parameters,
see~\cite[Section~2]{ENSTTV20} and also Example~\ref{ex:optimality} below in the more general setting of divergence-type
operators.

\begin{rmk}[Fractional Laplacian]\label{rmk:fractionalLaplacian}
	Using the transformation formula for spectral measures, see, e.g.,~\cite[Proposition~4.24]{Schm12}, the above corollary can be
	extended to cover also fractional Laplacians $(-\Delta_\Omega)^s$, $s > 0$. Indeed, for these operators we have the identity
	$\EE_{(-\Delta_\Omega)^s}(\lambda) = \EE_{-\Delta_\Omega}(\lambda^{1/s})$ for $\lambda \geq 0$, leading to a Bernstein-type
	inequality~\eqref{eq:Bernstein} with $C_B(m,\lambda) = \lambda^{m/s}$ for $\lambda \geq 0$. Consequently, the conclusion of
	Corollary~\ref{cor:pureLaplacian} holds also for these fractional Laplacians, but with $\sqrt{\lambda}$ replaced by
	$\lambda^{1/(2s)}$. The latter can also be deduced directly from Corollary~\ref{cor:pureLaplacian} via the transformation
	formula. Such a reasoning has previously been used in~\cite[Theorem~4.6]{NakicTTV-COVC}; see also~\cite[Remark~5.9]{ENSTTV20}.
	With regard to control theory and the particular form~\eqref{eq:spectralIneqControl} of a spectral inequality, it should,
	however, be mentioned that this is useful only if $1/(2s) < 1$, that is, $s > 1/2$. In fact, for $s \leq 1/2$ it is known that
	corresponding control properties of the fractional heat equation associated to $(-\Delta_\Omega)^s$ do not hold in general, see,
	e.g.,~\cite{MZ:06} and~\cite[Theorem 3]{koenig:18}.
\end{rmk}

\subsubsection{Divergence-type operators}\label{subsubsec:Divergence}

Let $H = H_\Omega(A)$ with some constant positive definite symmetric matrix $A \in \RR^{d\times d}$ be the Dirichlet or Neumann
realisation of the divergence-type differential expression
\begin{equation*}
 - \nabla \cdot (A\nabla)
\end{equation*}
as a non-negative self-adjoint operator on $L^2(\Omega)$ defined via its quadratic form. Denote by $\sigma_{\text{min}} > 0$ the
smallest eigenvalue of $A$. In this framework, we establish in Corollary~\ref{cor:bernstein-div-type} below the Bernstein-type
inequality~\eqref{eq:Bernstein} with the constant $C_B(m,\lambda) = (\sigma_{\text{min}}^{-1}\lambda)^m$, $\lambda \ge 0$, for
the following situations: 

\begin{itemize}

 \item $\Omega = \RR^d$ or a half-space;

 \item $A$ is a diagonal matrix and $\Omega = \bigtimes_{j=1}^d (a_j,b_j)$ a generalised rectangle as in the previous subsection.

\end{itemize}
This is essentially accomplished by transforming $H$ to a pure Laplacian and using the results discussed in the previous
subsection; the latter also explains the appearance of the additional factor $\sigma_{\min}^{-1}$ in the base of the
constants $C_B(m,\lambda)$, a consequence of a corresponding rescaling. The resulting quantity $h(\lambda)$
from~\eqref{eq:BernsteinSum:intro} reads
\begin{equation*}
 h(\lambda) = \exp(10\norm{l}{1}\sqrt{\lambda/\sigma_{\text{min}}}).
\end{equation*}
Combining this with the corresponding coverings from Lemma~\ref{lem:covering} and the asymptotic formula for $\tau_d$, we
obtain from Theorem~\ref{thm:LS} the following result.

\begin{cor}\label{cor:divergencetype}
 Let $A$ and $\Omega$ be as above, and let $\omega$ be a $(\gamma,\rho)$-thick set in $\Omega$. Then, for all $\lambda \geq 0$
 and all $f \in \Ran\EE_{H_\Omega(A)}(\lambda)$ we have 
 \begin{equation*}
  \norm{f}{L^2(\Omega)}^2
  \le
  \Bigl( \frac{K^d}{\gamma} \Bigr)^{Kd\rho\sqrt{\lambda / \sigma_{\min}} + 2d + 6} \norm{f}{L^2(\omega)}^2,
 \end{equation*}
 where $K > 0$ is a universal constant, and $\sigma_{\mathrm{min}}$ denotes the minimal eigenvalue of the positive definite
 matrix $A$.
\end{cor}

It is interesting to compare the above statement with the result in~\cite{BTV17}, where a divergence-type operator with lower
order terms and uniformly elliptic matrix function $A$ with small Lipschitz constant is considered on $\RR^d$ and certain
hypercubes. The authors give a spectral inequality for small energy intervals and subsets $\omega$ of the form
$\omega = \Omega \cap S$, where $S \subset \RR^d$ is an equidistributed set, that is, a union of balls $B(x_j, \delta)$ with
$B(x_j,\delta)\subset (-G/2, G/2)^d+ j$, $j\in(G\ZZ)^d$. For simplicity, we only discuss the whole space $\Omega = \RR^d$ and the
case $G = 1$ here.

Applying~\cite[Theorem~10]{BTV17} for our current divergence-type operators $H_{\RR^d}(A)$ with constant positive definite
matrix $A$ (and without lower order terms) gives
\begin{equation}\label{eq:BTV}
 \norm{f}{L^2(\RR^d)}^2\leq D_1\Big(\frac{D_2}{\delta}\Big)^{D_3(1+\lambda^{2/3})} \norm{f}{L^2(\omega)}^2
\end{equation}
for all $f \in \Ran \EE_{H_{\RR^d}(A)}([\lambda-\epsilon, \lambda+\epsilon])$, where
$\epsilon = D_1^{1/2}(D_2/\delta)^{(D_3/2)(1+\lambda^{2/3})}$ and $D_1, D_2, D_3$ are constants depending only on the dimension
$d$ and on the quantity $\max\{ \sigma_{\min}^{-1}, \sigma_{\max} \}$ with $\sigma_{\min}$, $\sigma_{\max}$ being the smallest
and largest eigenvalue of $A$, respectively.

Since equidistributed sets with $G = 1$ are $(\tau_d \delta^d, 2)$-thick in $\RR^d$, Corollary~\ref{cor:divergencetype} gives
instead
\[
 \norm{f}{L^2(\RR^d)}^2
 \le
 \Big(\frac{K^d}{\tau_d \delta^d}\Big)^{2Kd\sqrt{\lambda/\sigma_{\min}} +2d + 6 } \norm{f}{L^2(\omega)}^2
\]
for all $f \in \Ran \EE_{H_{\RR^d}(A)}(\lambda)$. The latter has several advantages over~\eqref{eq:BTV}: it does not depend on
$\sigma_{\max}$, depends on $\lambda$ only via $\lambda^{1/2}$ (as compared to $\lambda^{2/3}$), and is not restricted to short
energy intervals.

In fact, the term $\gamma^{-\sqrt{\lambda / \sigma_{\min}}}$ in the bound from Corollary~\ref{cor:divergencetype} is optimal, as
is shown by the following example, which is essentially taken from~\cite[Example~2.6]{ENSTTV20} and adapted to our present
situation; see also~\cite{Kovrijkine:01}.

\begin{ex}\label{ex:optimality}
 Let $A=\diag(\sigma_1, \ldots, \sigma_d)$ with $\sigma_d \ge \dots \ge \sigma_1 = \sigma_{\min} > 0$. For $f \in L^2(\RR^d)$, we
 have $f \in \Ran \EE_{H_{\RR^d}(A)}(\lambda)$ for some $\lambda>0$ if and only if $f$ satisfies
 $\supp \hat{f} \subset \cE_{\lambda} := \{ x\in\RR^d \colon \sum_{j=1}^d \sigma_j x_j^2 \leq \lambda \}$, where $\hat{f}$
 denotes the Fourier transform of $f$.

 For $\lambda > 0$ such that $\alpha := \sqrt{\lambda} / (2\pi \sqrt{d\sigma_{\min}}) \geq 2$ is a natural number, define
 $g \in L^2(\RR)$ by
 \[
  g(t):= \Big(\frac{\sin(2\pi t)}{t}\Big)^\alpha.
 \]
 Then, we have $\supp \hat{g} = [ -2\pi\alpha , 2\pi\alpha ]$. Furthermore, choose $h \in L^2(\RR^{d-1})$ with
 $\supp \hat{h} \subset \bigtimes_{j=2}^d [ -\sqrt{\lambda}/(\sqrt{d\sigma_j}) , \sqrt{\lambda}/(\sqrt{d\sigma_j} ]$, and define
 the function $f \in L^2(\RR^d)$ by $f(x) := g(x_1)h(x_2,\dots,x_d)$. Then,
 \begin{equation*}
   \supp \hat{f}
   \subset
   \bigtimes_{j=1}^d \Bigl[ -\frac{\sqrt{\lambda}}{\sqrt{d\sigma_j}} , \frac{\sqrt{\lambda}}{\sqrt{d\sigma_j}} \Bigr]
   \subset
   \cE_\lambda
   ,
 \end{equation*}
 so that $f \in \Ran \EE_{H_{\RR^d}(A)}(\lambda)$.

 Given $\gamma \in (0,1)$, let $\omega := \omega_1 \times \RR^{d-1}$ with a $1$-periodic set $\omega_1\subset\RR$ such that
 $\omega_1 \cap [-1/2, 1/2]= [-1/2, (-1+\gamma)/2]\cup [(1-\gamma)/2, 1/2]$. Then, $\omega_1$ is $(\gamma,1)$-thick in $\RR$
 and, thus, $\omega$ is $(\gamma, 1)$-thick in $\RR^d$.

 With $2\alpha-2\geq 2$ and $\norm{g}{L^2(\RR)} \ge 1$, a calculation as in~\cite[Example~2.6]{ENSTTV20} shows that
 \begin{equation*}
   \norm{g}{L^2(\omega_1)}^2
   \le
   2\pi^2 (3\pi\gamma)^{2\alpha - 2}
   \le
   (3\sqrt{2}\pi^2\gamma)^{2\alpha - 2} \norm{g}{L^2(\RR)}^2
   .
 \end{equation*}
 By separation of variables and together with Corollary~\ref{cor:divergencetype}, this yields
 \[
  \Big(\frac{\gamma}{K^d}\Big)^{Kd \sqrt{\frac{\lambda}{\sigma_{\min}}}+2d +5} \norm{f}{L^2(\RR^d)}^2 
  \leq
  \norm{f}{L^2(\omega)}^2
  \leq
  \Big(\frac{\gamma}{1/(3\sqrt{2}\pi^2)}\Big)^{\frac{1}{\pi}\sqrt{\frac{\lambda}{d \sigma_{\min}}}-2} \norm{f}{L^2(\RR^d)}^2
  .
 \]
 Hence, the term $\gamma^{\sqrt{\lambda/\sigma_{\min}}}$ in this bound is indeed optimal.
\end{ex}

\subsubsection{The harmonic oscillator}\label{subsubsec:harmonicOscillator}
Let $H_\Omega = -\Delta_\Omega + \abs{x}^2$ be the Dirichlet or Neumann realisation of the differential expression
\begin{equation*}
 -\Delta + \abs{x}^2
\end{equation*}
as a non-negative self-adjoint operator on $L^2(\Omega)$ defined as a form sum. In this framework, we establish in
Proposition~\ref{prop:harmonicOscillator} below the Bernstein-type inequality~\eqref{eq:Bernstein} with
\begin{equation*}
 C_B(m,\lambda)
 =
 (2\delta)^{2m} \ee^{\ee/\delta^2} (m!)^2 \ee^{2\sqrt{\lambda}/\delta}
 ,\quad
 \lambda \ge 0
 ,
\end{equation*}
where $\delta > 0$ is a fixed chosen parameter, for the following domains:

\begin{itemize}

 \item generalised rectangles $\bigtimes_{j=1}^d (a_j,b_j)$ with $a_j,b_j \in \{ 0,\pm\infty \}$, $a_j < b_j$;

 \item the sector $S_{\pi/4} \subset \RR^2$;

 \item finite Cartesian products of $S_{\pi/4}$ and generalised rectangles of the above form.

\end{itemize}
This builds upon and extends the case of the whole space $\Omega = \RR^d$ studied in~\cite{BJPS18}; note that here each spectral
subspace $\Ran\EE_{H_{\RR^d}}(\lambda)$ describes certain finite linear combinations of the standard Hermite functions. The
resulting quantity $h(\lambda)$ from~\eqref{eq:BernsteinSum:intro} now reads
\begin{equation*}
 h(\lambda)
 =
 \ee^{\ee/(2\delta^2)} \ee^{\sqrt{\lambda}/\delta} \sum_{m \in \NN_0} (20\delta\norm{l}{1})^m
 .
\end{equation*}
Choosing $\delta := 1/(40\norm{l}{1})$ leads to $20\delta\norm{l}{1} = 1/2$, and consequently to 
\begin{equation*}
 \log h(\lambda)
 =
 \log 2 + \frac{\ee}{2\delta^2} + \frac{\sqrt{\lambda}}{\delta}
 =
 \log 2 + 800\ee\norm{l}{1}^2 + 40\norm{l}{1}\sqrt{\lambda}
 .
\end{equation*}
Combining this again with Lemma~\ref{lem:covering} and the asymptotic formula for $\tau_d$, we obtain from Theorem~\ref{thm:LS} the
following statement for the harmonic oscillator.

\begin{cor}\label{cor:harmonicOscillator}
 Let $H_\Omega$ be the harmonic oscillator on one of the above mentioned domains. Then, we have
 \begin{equation*}
  \norm{f}{L^2(\Omega)}^2
  \le
  \Bigl( \frac{K^d}{\gamma} \Bigr)^{Kd\rho\sqrt{\lambda} + K^2d^2\rho^2 + 2d + 10} \norm{f}{L^2(\omega)}^2
 \end{equation*}
 for all $f \in \Ran\EE_{H}(\lambda)$, $\lambda \ge 0$, and every $(\gamma,\rho)$-thick set $\omega$ in $\Omega$, where $K > 0$ is
 a universal constant. 
\end{cor}

It is interesting to note that in~\cite{BJPS18} the authors obtain for the case $\Omega = \RR^d$ the constant
\[
 \frac{8}{\gamma}\Big(\frac{K^d}{\gamma}\Big)^{\frac{C_1(d)\rho\sqrt{\lambda}}{\log 2} +\frac{2\log((4\rho)^{d/2})C_2(d)}{\log 2}}
\] 
with $C_1(d), C_2(d) > 0$ depending on the dimension $d$. Although this is essentially compatible with the corollary above, 
it shows some discrepancies, mainly in the dependence on $\rho$, and the dependence on the dimension $d$ is not explicit. This is
due to the technical differences in the proof. Indeed, while the proof in~\cite{BJPS18} is again inspired by the work of
Kovrijkine and essentially follows the same strategy as our proof presented in Section~\ref{sec:main_proof} below, some of the
arguments are carried out differently. For example, a different definition for good and bad elements of the covering
(cf.~Section~\ref{subsec:good_and_bad} below) is used, as well as a slightly different approach to the proof of the
local estimate in Lemma~\ref{lem:local_estimate} below making use of a Sobolev embedding.

\section{Bernstein-type inequalities}\label{sec:bernstein}

In this section we discuss approaches to verify Bernstein-type inequalities of the form~\eqref{eq:Bernstein}. Some of these
approaches are direct (Section~\ref{subsec:bernstein-criterion}) and some derive such inequalities from other situations, for
instance, by taking Cartesian products (Section~\ref{subsec:Cartesian}) or by applying linear transformations
(Section~\ref{subsec:transformations}). Finally, we combine these considerations in Section~\ref{subsec:specific} to prove the
Bernstein-type inequalities for the situations mentioned in Section~\ref{subsec:applications}.

In the following, it is convenient to mostly avoid the reference to $H$ and $\lambda$ in the notations. This makes it, however,
necessary to slightly adapt the notion of a Bernstein-type inequality: We say that a function
$f \in W^{\infty,2}(\Omega) = \bigcap_{k\in\NN} W^{k,2}(\Omega)$ satisfies a~\emph{Bernstein-type inequality with respect to
$C_B \colon \NN_0 \to (0,\infty)$} if
\begin{equation*}
 \sum_{\abs{\alpha} = m} \frac{1}{\alpha!}\norm{\partial^\alpha f}{L^2(\Omega)}^2 \leq \frac{C_B(m)}{m!} \norm{f}{L^2(\Omega)}^2
 \quad\text{ for all }\quad m\in\NN_0.
\end{equation*}

\subsection{A criterion for Bernstein}\label{subsec:bernstein-criterion}

The following result applies to certain classes of functions in $W^{\infty,2}(\Omega)$ and allows us to check directly the
validity of Bernstein-type inequalities in the context of spectral subspaces for the pure Laplacian. It also provides sharp
Bernstein constants in these situations.

\begin{lemma}\label{lem:Bernstein-Criterion}
	Let $\cM \subset W^{\infty,2}(\Omega)$ be a subset that is invariant for the differential expression $-\Delta$, that is,
	$-\Delta f \in \cM$ for all $f\in \cM$. Furthermore, suppose that
	\begin{equation}\label{eq:BernsteinCriterion:cond}
	\sum_{\abs{\alpha}=n} \frac{1}{\alpha!} \sum_{j=1}^d \langle \partial_j \partial^\alpha f,
	\partial_j \partial^\alpha g \rangle_{L^2(\Omega)}
	=
	\sum_{\abs{\alpha}=n} \frac{1}{\alpha!} \langle \partial^\alpha f,(-\Delta)\partial^\alpha g \rangle_{L^2(\Omega)}
	\end{equation}
	for all $n \in \NN_0$ and all $f,g \in \cM$. Then,
	\begin{equation}\label{eq:bernstein-criterion}
	\sum_{\abs{\alpha}=m} \frac{1}{\alpha!} \langle \partial^\alpha f, \partial^\alpha g \rangle_{L^2(\Omega)}
	=
	\frac{1}{m!} \langle f,(-\Delta)^m g \rangle_{L^2(\Omega)}
	\quad \forall\, f,g\in\cM \quad \forall\, m\in\NN_0
  .
	\end{equation}	
\end{lemma}

\begin{proof}
	We proceed by induction. The case $m=0$ is clear, and the case $m=1$ corresponds to the hypothesis for $n=0$. Suppose that the
	claim holds for some $m \in \NN$. Since $-\Delta g \in \cM$ by hypothesis, we then have
	\begin{align*}
	\frac{1}{(m+1)!} \langle f,(-\Delta)^{m+1} g \rangle_{L^2(\Omega)}
	&=
	\frac{1}{(m+1)!} \langle f,(-\Delta)^m (-\Delta g) \rangle_{L^2(\Omega)}\\
	&=
	\frac{1}{m+1} \sum_{\abs{\beta}=m} \frac{1}{\beta!} \langle \partial^\beta f,\partial^\beta(-\Delta g) \rangle_{L^2(\Omega)}\\
	&=
	\frac{1}{m+1} \sum_{\abs{\beta}=m} \frac{1}{\beta!} \langle \partial^\beta f,(-\Delta)\partial^\beta g \rangle_{L^2(\Omega)}\\
	&=
	\frac{1}{m+1} \sum_{\abs{\beta}=m} \frac{1}{\beta!} \sum_{j=1}^d \langle \partial_j\partial^\beta f,\partial_j\partial^\beta g
	\rangle_{L^2(\Omega)}\\
	&=
	\sum_{\abs{\alpha}=m+1} \frac{1}{\alpha!} \langle \partial^\alpha f,\partial^\alpha g \rangle_{L^2(\Omega)},
	\end{align*}
	which proves the claim.
\end{proof}%

It is worth to note that the formulation of the above statement avoids any regularity assumption on the domain $\Omega$. If,
however, $\Omega$ is sufficiently regular, the hypothesis~\eqref{eq:BernsteinCriterion:cond} of
Lemma~\ref{lem:Bernstein-Criterion} may be verified by taking care of suitable boundary integrals:

\begin{rmk}\label{rmk:Bernstein}
	If $\Omega$ admits an integration by parts formula (for example, if $\Omega$ has Lipschitz boundary),
	equation~\eqref{eq:BernsteinCriterion:cond} is, by Green's formula, equivalent to
	\begin{equation}\label{eq:boundary-integration}
		\sum_{\abs{\alpha}=m} \frac{1}{\alpha!}
		\int_{\partial\Omega}\partial^\alpha f(x)\frac{\partial}{\partial\nu}\partial^\alpha \overline{g(x)}
		\dd\sigma(x)=0 
	\end{equation} 
	for all $m\in\NN_0$ and all $f,g \in \cM$.
\end{rmk}

\subsection{Cartesian products}\label{subsec:Cartesian}

In this subsection, we discuss situations in which relation~\eqref{eq:bernstein-criterion} can be lifted to Cartesian products.

Consider $\Omega = \Omega_1 \times \Omega_2 \subset \RR^d$ with domains $\Omega_1 \subset \RR^{d_1}$ and
$\Omega_2 \subset \RR^{d_2}$, $d_1 + d_2 = d$. We assume that at least one of the two Laplacians $-\Delta_{\Omega_1}$ and
$-\Delta_{\Omega_2}$ has purely discrete spectrum, say $-\Delta_{\Omega_1}$. This holds, for instance, if $\Omega_1$ is a bounded
Lipschitz domain.

Using the language of tensor products (see, e.g.,~\cite[Section~7.5]{Schm12} and also~\cite[Section~8.5]{Wei80}), the product
structure of $\Omega$ guarantees that
\begin{equation}\label{eq:tensor}
 -\Delta_\Omega = (-\Delta_{\Omega_1}) \otimes I_2 + I_1 \otimes (-\Delta_{\Omega_2}),
\end{equation}
where $I_j$ denotes the identity operator on $L^2(\Omega_j)$, $j\in\{1,2\}$. Set
\begin{equation*}
 \cM := \bigcup_{\lambda \ge 0} \Ran\EE_{-\Delta_{\Omega}}(\lambda)
 \quad\text{ and }\
 \cM_j := \bigcup_{\lambda \ge 0} \Ran\EE_{-\Delta_{\Omega_j}}(\lambda), \ j \in \{1,2\}
 .
\end{equation*}

With this notational setup, we obtain the following result, a variant of which (for $m=2$) has previously been discussed
in~\cite[Proposition~4.1]{Seel20}; cf.~also Remark~\ref{rmk:cartesian} below.

\begin{lemma}\label{lemma:bernstein-cartesian-product}
  If for $j \in \{1,2\}$ we have $\cM_j \subset W^{\infty,2}(\Omega_j)$ with
	\begin{equation*}
	\sum_{\abs{\alpha} = m} \frac{1}{\alpha!} \langle \partial^\alpha f, \partial^\alpha g \rangle_{L^2(\Omega_j)}
	=
	\frac{1}{m!} \langle f , (-\Delta_{\Omega_j})^m g \rangle_{L^2(\Omega_j)}
	\quad \forall\, f,g\in\cM_j \quad \forall\, m\in\NN_0,
	\end{equation*}
	then $\cM \subset W^{\infty,2}(\Omega)$ and
	\begin{equation*}
	\sum_{\abs{\alpha} = m} \frac{1}{\alpha!} \langle \partial^\alpha f, \partial^\alpha g \rangle_{L^2(\Omega)}
	=
	\frac{1}{m!} \langle f , (-\Delta_\Omega)^m g \rangle_{L^2(\Omega)}
	\quad \forall\,f,g\in\cM \quad\forall\,m\in\NN_0.
	\end{equation*}
\end{lemma}

\begin{proof}
 We first consider functions $f = \phi \otimes \psi, g = \tilde{\phi} \otimes \tilde{\psi} \in W^{\infty,2}(\Omega)$ with
 $\phi,\tilde{\phi} \in \cM_1$ and $\psi,\tilde{\psi} \in \cM_2$. For all $m \in \NN_0$, we then have
 \begin{align*}
  \sum_{\abs{\alpha} = m} \frac{1}{\alpha!} \langle \partial^\alpha f , \partial^\alpha g &\rangle_{L^2(\Omega)}
  =
  \sum_{k=0}^m
  \sum_{\abs{\beta}=k} \frac{1}{\beta!} \langle \partial^\beta \phi , \partial^\beta \tilde{\phi} \rangle_{L^2(\Omega_1)}
  \sum_{\abs{\gamma}=m-k} \frac{1}{\gamma!} \langle \partial^\gamma \psi , \partial^\gamma \tilde{\psi} \rangle_{L^2(\Omega_2)}\\
  &=
  \sum_{k=0}^m \frac{1}{k!} \langle \phi , (-\Delta_{\Omega_1})^k \tilde{\phi} \rangle_{L^2(\Omega_1)}
  \frac{1}{(m-k)!} \langle \psi , (-\Delta_{\Omega_2})^{m-k} \tilde{\psi}\rangle_{L^2(\Omega_2)}\\
  &=
  \frac{1}{m!} \sum_{k=0}^m \binom{m}{k} \langle f , (-\Delta_{\Omega_1})^k \tilde{\phi} \otimes
  (-\Delta_{\Omega_2})^{m-k} \tilde{\psi} \rangle_{L^2(\Omega)}\\
  &=
  \frac{1}{m!} \langle f , (-\Delta_{\Omega})^m g \rangle_{L^2(\Omega)}
  ,
 \end{align*}
 where for the last equality we have taken into account~\eqref{eq:tensor} and~\cite[Eq.~(7.44)]{Schm12}.

 In order to conclude the proof, by sesquilinearity it now suffices to show that each function in $\cM$ is a finite sum of
 elementary tensors of the above form. To this end, let $f \in \EE_{-\Delta_\Omega}(\lambda)$ for some $\lambda \ge 0$, and
 recall that $-\Delta_{\Omega_1}$ is assumed to have purely discrete spectrum. Due to the tensor structure~\eqref{eq:tensor} of
 $-\Delta_\Omega$, we then have
 \begin{equation}\label{eq:tensorSpecMeas}
  \EE_{-\Delta_{\Omega}}(\lambda)= \sum_{\mu \leq \lambda} P_{\mu} \otimes \EE_{-\Delta_{\Omega_2}}(\lambda-\mu),
 \end{equation}
 where $P_{\mu}$ is the orthogonal projection onto the kernel of $-\Delta_{\Omega_1}-\mu$, see~\cite[Theorem 8.34 and
 Exercise~8.21]{Wei80}. Writing $f(x,y) = \sum_k \langle f(\cdot,y) , \phi_k \rangle_{L^2(\Omega_1)} \phi_k(x)$ with an
 orthonormal family $(\phi_k)_k$ of eigenfunctions for $-\Delta_{\Omega_1}$ with corresponding eigenvalues $\lambda_k$, we then
 infer from~\eqref{eq:tensorSpecMeas} that
 \begin{equation*}
   f = \sum_{\lambda_k \le \lambda} \phi_k \otimes \psi_k
   \quad\text{ with suitable }\
   \psi_k \in \Ran\EE_{-\Delta_{\Omega_2}}(\lambda-\lambda_k) \subset \cM_2
   ,
 \end{equation*}
 which completes the proof.
\end{proof}%

\begin{rmk}\label{rmk:cartesian}
 Note that for $m = 2$ the identities in Lemma~\ref{lemma:bernstein-cartesian-product} can be written as
 \begin{equation*}
 	\sum_{\abs{\alpha}=2} \frac{1}{\alpha!} \langle \partial^\alpha f , \partial^\alpha g \rangle_{L^2(\Omega)}
 	=
 	\frac{1}{2} \langle (-\Delta_\Omega)f , (-\Delta_\Omega)g \rangle_{L^2(\Omega)}
 \end{equation*}
 for $f,g \in \cM$, and analogously for $-\Delta_{\Omega_j}$, which clearly extend to all elements $f,g$ from the respective
 operator domains by approximation. In particular, the domain of $\Delta_\Omega$ belongs to $W^{2,2}(\Omega)$, and the Sobolev
 norm of second order for elements of the domain can be estimated by the graph norm of $\Delta_\Omega$. This is the essence
 of~\cite[Proposition~4.1]{Seel20}. The corresponding proof there is simpler since one does not have to revert to the subspaces
 $\cM_j \subset W^{\infty,2}(\Omega_j)$ and $\cM \subset W^{\infty,2}(\Omega)$, respectively. As a consequence, one does also not
 have to rely on purely discrete spectrum of one of the $-\Delta_{\Omega_j}$.
\end{rmk}

\subsection{Transformation of Bernstein sums}\label{subsec:transformations}

In the following, we identify a matrix in $\RR^{d\times d}$ with the corresponding induced linear mapping on $\RR^d$ with respect
to the standard basis.

\begin{lemma}\label{lem:scaling}
	Let $P\in\RR^{d\times d}$ be symmetric and positive definite, and suppose that the function
	$f \circ P \in W^{\infty,2}(P^{-1}(\Omega))$ satisfies a Bernstein-type inequality with respect to
	$\tilde{C}_B \colon \NN_0 \to (0,\infty)$. Then, $f \in W^{\infty,2}(\Omega)$ satisfies a Bernstein-type inequality with respect
	to 	$C_B \colon \NN_0 \to (0,\infty)$ given by $C_B(m) = \tilde{C}_B(m) / p_{\min}^{2m}$, where $p_{\min}$ is the smallest
	eigenvalue of $P$.
\end{lemma}

\begin{proof}
	It follows from Lemma~\ref{lemma:derivative-transformation}\,(b) that for each $m \in \NN_0$ we have
	\begin{align*}
		\sum_{\abs{\alpha}=m}\frac{1}{\alpha!}\norm{\partial^\alpha f}{L^2(\Omega)}^2  
		&\leq
		\frac{\det P}{p_{\min}^{2m}} \sum_{\abs{\alpha}=m}\frac{1}{\alpha!}	\norm{\partial^\alpha (f \circ P)}{L^2(P^{-1}(\Omega))}^2\\
	 	&\leq
	 	\frac{\tilde{C}_B(m)\det P}{p_{\min}^{2m}m!} \norm{f \circ P}{L^2(P^{-1}(\Omega))}^2
		=
		\frac{\tilde{C}_B(m)}{p_{\min}^{2m}m!}\norm{f}{L^2( \Omega)}^2
		.\qedhere
	\end{align*}
\end{proof}%

We now consider the extension of functions by single or multiple reflections:

Let $M \colon \RR^d \to \RR^d$ be the reflection with respect to any hyperplane in $\RR^d$. Let further
$\Omega,\tilde{\Omega} \subset \RR^d$ be domains such that $\tilde{\Omega}$ is symmetric with respect to $M$ and $\Omega$ belongs
to one of the two (open) half-spaces defined by the reflection.
We define the extension operator $X \colon L^2(\Omega) \to L^2(\tilde{\Omega}) = L^2(\Omega) \oplus L^2(M(\Omega))$ by
\begin{equation*}
	Xf
	:=
	f \oplus \mu(f \circ M)
\end{equation*}
with $\mu = -1$ for an antisymmetric and $\mu = 1$ for a symmetric reflection. We now immediately obtain the following result.

\begin{lemma}\label{lemma:singeReflection}
	Let $M$, $\Omega$, $\tilde{\Omega}$, and $X$ be as above, and let $f,g \in L^2(\Omega)$. If both
	$Xf, Xg \in W^{\infty,2}(\tilde{\Omega})$,	then also $f,g \in W^{\infty,2}(\Omega)$, and for all $m \in \NN_0$ we have
	\begin{equation*}
		\sum_{\abs{\alpha}=m} \frac{1}{\alpha!} \langle \partial^\alpha(Xf) , \partial^\alpha(Xg) \rangle_{L^2(\tilde{\Omega})}
		=
		2 \sum_{\abs{\alpha}=m} \frac{1}{\alpha!} \langle \partial^\alpha f , \partial^\alpha g \rangle_{L^2(\Omega)}
		.
	\end{equation*}
	In particular, if $Xf \in W^{\infty,2}(\tilde{\Omega})$ satisfies a Bernstein-type inequality with respect to
	$C_B \colon \NN_0 \to (0,\infty)$, then also $f \in W^{\infty,2}(\Omega)$ does.
\end{lemma}

\begin{proof}
	Since $f = (Xf)|_\Omega$, we clearly have $f \in W^{\infty,2}(\Omega)$ if $Xf \in W^{\infty,2}(\tilde{\Omega})$, and the same
	holds for $g$ instead of $f$. It now follows from part~(a) of Lemma~\ref{lemma:derivative-transformation} that
	\begin{equation*}	
		\sum_{\abs{\alpha}=m} \frac{1}{\alpha!} \langle \partial^\alpha (f\circ M) , \partial^\alpha (g\circ M) \rangle_
		{L^2(M(\Omega))}
		=
		\sum_{\abs{\alpha}=m} \frac{1}{\alpha!} \langle \partial^\alpha f , \partial^\alpha g \rangle_{L^2(\Omega)}
	\end{equation*}
	for all $m \in \NN_0$. In light of $\partial^\alpha (Xf) = \partial^\alpha f \oplus (\partial^\alpha(f\circ M))$ for
	all $\alpha \in \NN_0^d$, this proves the stated identity.

	Finally, if $Xf \in W^{\infty,2}(\tilde{\Omega})$ satisfies a Bernstein-type inequality with respect to
	$C_B \colon \NN_0 \to (0,\infty)$, then for all $m \in \NN_0$ we have
	\begin{align*}
		\sum_{\abs{\alpha}=m}\frac{1}{\alpha!} \norm{\partial^\alpha f}{L^2(\Omega)}^2
		&= 
		\frac{1}{2}\sum_{\abs{\alpha}=m}\frac{1}{\alpha!} \norm{\partial^\alpha (Xf)}{L^2(\tilde \Omega)}^2\\
		&\leq
		\frac{1}{2} \frac{C_B(m)}{m!}\norm{Xf}{L^2(\tilde \Omega)}^2 = \frac{C_B(m)}{m!}\norm{f}{L^2(\Omega)}^2
		, 
	\end{align*}
	which completes the proof.
\end{proof}%

As an example of extensions by multiple reflections, we consider the successive reflection of a sector
$S_\theta = \{ (x_1,x_2) \in (0,\infty)^2 \colon x_2 < x_1\tan\theta \}\subset \RR^2$ of angle $\theta = \pi / n$, $n \in \NN$,
$n \geq 3$, to $\RR^2$:

Let $R \colon \RR^2 \to \RR^2$ be the rotation by angle $\theta = \pi / n$ around the origin, and let $M \colon \RR^2 \to \RR^2$
be the reflection with respect to the second coordinate. Define $M_k \colon \RR^2 \to \RR^2$ for $k = 1, \dots, 2n-1$ by
$M_k := R^k M R^{-k}$ and set $K_k := M_k \circ \dots \circ M_1$; note that $M_k$ represents the reflection with respect to
$R^k(\RR \times \{0\})$. With the identity
$L^2(\RR^2) = L^2(S_\theta) \oplus L^2(K_1(S_\theta)) \oplus \dots \oplus L^2(K_{2n-1}(S_\theta))$, we define the extension
operator $X \colon L^2(S_\theta) \to L^2(\RR^2)$ by
\begin{equation}\label{eq:multipleReflection}
	Xf
	:=
	f \oplus \mu(f \circ K_1^{-1}) \oplus \dots \oplus \mu^{2n-1}(f \circ K_{2n-1}^{-1})
\end{equation}
with $\mu = -1$ for antisymmetric and $\mu = 1$ for symmetric reflections. The following result formulates an analogue of
Lemma~\ref{lemma:singeReflection} in this situation. Its proof is very similar to the one of the latter (with each $K_k$ playing
the role of $M$) and is hence omitted here.

\begin{lemma}\label{lemma:multipleReflection}
	Let $X \colon L^2(S_\theta) \to L^2(\RR^2)$ be defined as in~\eqref{eq:multipleReflection}, and let $f,g \in L^2(S_\theta)$. If
	both $Xf,Xg \in W^{\infty,2}(\RR^2)$, then also $f,g \in W^{\infty,2}(S_\theta)$, and for all $m \in \NN_0$ we have
	\begin{equation*}
		\sum_{\abs{\alpha}=m} \frac{1}{\alpha!} \langle \partial^\alpha (Xf) , \partial^\alpha (Xg) \rangle_{L^2(\RR^2)}
		=
		2n \sum_{\abs{\alpha}=m} \frac{1}{\alpha!} \langle \partial^\alpha f , \partial^\alpha g \rangle_{L^2(S_\theta)}
		.
	\end{equation*}
\end{lemma}

The above can easily be modified to allow also Cartesian products with a fixed domain $\Omega \subset \RR^d$. More precisely, for
$k = 1, \dots, 2n-1$ set $K_k' := K_k \oplus \mathrm{Id}$, where $\mathrm{Id}$ denotes the identity mapping on $\RR^d$, and define
$X \colon L^2(S_\theta \times \Omega) \to L^2(\RR^2 \times \Omega)$ by
\begin{equation}\label{eq:multipleReflection:Cartesian}
	Xf
	:=
	f \oplus \mu (f \circ K_1'^{-1}) \oplus \dots \oplus \mu^{2n-1} (f \circ K_{2n-1}'^{-1})
\end{equation}
with $\mu \in \{\pm 1\}$ as above. For future reference, we formulate the following variant of
Lemma~\ref{lemma:multipleReflection}, the proof of which we likewise skip.

\begin{lemma}\label{lemma:multipleReflection:Cartesian}
	Let $X \colon L^2(S_\theta \times \Omega) \to L^2(\RR^2 \times \Omega)$ be defined as
	in~\eqref{eq:multipleReflection:Cartesian}, and let $f,g \in L^2(S_\theta \times \Omega)$. If
	$Xf,Xg \in W^{\infty,2}(\RR^2 \times \Omega)$, then 	also $f,g \in W^{\infty,2}(S_\theta \times \Omega)$, and for all
	$m \in \NN_0$	we have
	\begin{equation*}
		\sum_{\abs{\alpha}=m} \frac{1}{\alpha!} \langle \partial^\alpha (Xf) , \partial^\alpha (Xg) \rangle_
		{L^2(\RR^2 \times \Omega)}
		=
		2n \sum_{\abs{\alpha}=m} \frac{1}{\alpha!} \langle \partial^\alpha f , \partial^\alpha g \rangle_
		{L^2(S_\theta \times \Omega)}
		.
	\end{equation*}
\end{lemma}

\subsection{Applications}\label{subsec:specific}
We now use the above considerations to show Bernstein-type inequalities in the situations mentioned in
Section~\ref{subsec:applications}. Here, taking into account that the inequality
$\langle f , (-\Delta_\Omega)^m f \rangle \le \lambda^m \norm{f}{L^2(\Omega)}^2$ holds for all
$f \in \Ran_{-\Delta_\Omega}(\lambda)$, $\lambda \geq 0$, by functional calculus, the relation
\begin{equation}\label{eq:Bernstein:spectral:2}
	\sum_{\abs{\alpha}=m} \frac{1}{\alpha!} \langle \partial^\alpha f , \partial^\alpha g \rangle_{L^2(\Omega)}
	=
	\frac{1}{m!} \langle f , (-\Delta_\Omega)^m g \rangle_{L^2(\Omega)}
\end{equation}
for $f,g \in \Ran\EE_{-\Delta_\Omega}(\lambda)$, $\lambda \geq 0$, plays a crucial role. It can be established on certain domains
directly via Lemma~\ref{lem:Bernstein-Criterion}. Bernstein-type inequalities in other situations can then be deduced from this by
taking Cartesian products (Lemma~\ref{lemma:bernstein-cartesian-product}), by rescaling (Corollary~\ref{lem:scaling}), or by
single or multiple reflections (Lemmas~\ref{lemma:singeReflection},~\ref{lemma:multipleReflection},
and~\ref{lemma:multipleReflection:Cartesian}). The latter requires that certain transformed functions $Xf$ with
$f \in \Ran\EE_{H_\Omega}(\lambda)$ are elements of a spectral subspace for the pure Laplacian on some other domain
$\tilde{\Omega}$ where relation~\eqref{eq:Bernstein:spectral:2} is known to hold or, at least, some Bernstein-type inequality is
available. An abstract framework for this is provided by the following result extracted from~\cite[Lemma~3.3]{ES19}.

\begin{lemma}\label{lemma:spectral-relation}
	Let $H$ and $\tilde H$ be lower semibounded self-adjoint operators on Hilbert spaces $\cH$ and $\tilde{\cH}$, respectively, and
	let $X \colon \cH\rightarrow \tilde{\cH}$ be a bounded linear operator such that $\tilde H X \supset X H$, that is,
	$Xf \in \Dom(\tilde{H})$ and $\tilde{H} Xf = X Hf$ for all $f \in \Dom(H)$. Then, the spectral families $\EE_{\tilde H}$	and
	$\EE_H$ associated with $\tilde H$ and $H$, respectively, satisfy
	\begin{equation}\label{eq:spectral-relation}
		\EE_{\tilde H}(\lambda) X = X \EE_H (\lambda) \qquad \forall\, \lambda\in\RR.
	\end{equation}
\end{lemma}

Note that relation~\eqref{eq:spectral-relation} indeed ensures that $Xf \in \Ran\EE_{\tilde{H}}(\lambda)$ if
$f \in \Ran\EE_H(\lambda)$. The required extension relation $\tilde H X \supset X H$ can usually be conveniently verified in terms
of the associated forms:
Let $\fh$ and $\tilde{\fh}$ denote the sesquilinear forms associated with $H$ and $\tilde{H}$, respectively. Then, the relation
$\tilde H X \supset X H$ holds if
\begin{equation}\label{eq:extRelationForm}
	\tilde{\fh}[Xf , g]
	=
	\fh[ f , X^*g ]
	\qquad \forall\, f \in \Dom(\fh),\ \forall\, g \in \Dom(\tilde{\fh})
	,
\end{equation}
which implicitly requires that $X$ maps $\Dom(\fh)$ into $\Dom(\tilde{\fh})$ and $X^*$ maps $\Dom(\tilde{\fh})$ into $\Dom(\fh)$.
Although this approach formally allows to handle fairly general operators $H$ and $\tilde{H}$ in the context of
Lemma~\ref{lemma:spectral-relation}, we treat below only those specific cases of interest for our intended applications.

We first treat the case of single reflections: Let $M$ and $X \colon L^2(\Omega) \to L^2(\tilde{\Omega})$ be as in
Lemma~\ref{lemma:singeReflection}, and consider the pure Laplacians $-\Delta_\Omega$ and $-\Delta_{\tilde{\Omega}}$, both sharing
the same type of boundary conditions (Dirichlet or Neumann). The desired relation
$(-\Delta_{\tilde{\Omega}})X \supset X(-\Delta_\Omega)$ then follows from the more general considerations
in~\cite[Section~4]{ES19}, where the corresponding relation~\eqref{eq:extRelationForm} has been established for uniformly elliptic
divergence-type operators including an additive bounded potential (upon a suitable extension of the coefficients to
$\tilde{\Omega}$). In particular, it has been shown in~\cite[Lemma~4.1]{ES19} that $X$ and $X^*$ map the form domains of
$-\Delta_\Omega$ and $-\Delta_{\tilde{\Omega}}$ (that is, $H_0^1$ or $H^1$, depending on the boundary conditions) appropriately as
indicated in~\eqref{eq:extRelationForm}. Hence, Lemma~\ref{lemma:spectral-relation} can be applied in this situation.
Consequently, if $Xf,Xg \in \Ran_{-\Delta_{\tilde{\Omega}}}(\lambda)$ with $f,g \in \Ran\EE_{-\Delta_\Omega}(\lambda)$ satisfy
relation~\eqref{eq:Bernstein:spectral:2} on $\tilde{\Omega}$, then by Lemmas~\ref{lemma:singeReflection}
and~\ref{lemma:spectral-relation} we have
\begin{equation}\label{eq:Bernstein:spectral:singleReflection}
	\begin{aligned}
		\frac{1}{m!}\langle f , (-\Delta_\Omega)^m g \rangle_{L^2(\Omega)}
		&=
		\frac{1}{2m!} \langle Xf , X(-\Delta_\Omega)^m g \rangle_{L^2(\tilde\Omega)}\\
		&=
		\frac{1}{2m!} \langle Xf , (-\Delta_{\tilde\Omega})^m Xg \rangle_{L^2(\tilde\Omega)}\\
		&=
		\frac{1}{2} \sum_{\abs{\alpha}=m} \frac{1}{\alpha!} \langle \partial^\alpha (Xf) , \partial^\alpha (Xg) \rangle_
		{L^2(\tilde\Omega)}\\
		&=
		\sum_{\abs{\alpha}=m} \frac{1}{\alpha!}\langle \partial^\alpha f , \partial^\alpha g \rangle_{L^2(\Omega)}
	\end{aligned}
\end{equation}
for all $m \in \NN_0$, so that $f,g$ satisfy relation~\eqref{eq:Bernstein:spectral:2} on $\Omega$.

The considerations in~\cite[Section~4]{ES19} used to prove~\eqref{eq:extRelationForm} for single reflections mainly involve
integration by parts and change of variables. It is not hard to adapt them to the case of multiple reflections as in
Lemmas~\ref{lemma:multipleReflection} and~\ref{lemma:multipleReflection:Cartesian}. We will not perform this adaptation here and
just take the corresponding relation~\eqref{eq:extRelationForm} for granted, so that the conclusion of
Lemma~\ref{lemma:spectral-relation} will be available also for these multiple reflections. The corresponding variant
of~\eqref{eq:Bernstein:spectral:singleReflection} (with $1/2$ replaced by $(1/2n)$) then holds as well.

We are now ready to prove relation~\eqref{eq:Bernstein:spectral:2} for the domains mentioned in Section~\ref{subsubsec:Laplacian}.

\begin{prop}\label{prop:bernstein-pureLaplacian}
	Let $\Omega$ be one of the following domains:
	\begin{enumerate}[(i)]

		\item a generalised rectangle $\bigtimes_{j=1}^d (a_j, b_j)\subset\RR^d$ with $-\infty \leq a_j < b_j\leq \infty$ for all
					$j=1,\ldots,d$;

		\item an equilateral triangle in $\RR^2$;

		\item an isosceles right-angled or an hemiequilateral triangle in $\RR^2$;

		\item a sector $S_\theta \subset \RR^2$ with angle $\theta = \pi / n$, $n \in \NN$, $n \geq 3$;

		\item any finite Cartesian product of the above.

	\end{enumerate}

	Then, $\Ran\EE_{-\Delta_\Omega}(\lambda) \subset W^{\infty,2}(\Omega)$ for all $\lambda \geq 0$,
	and relation~\eqref{eq:Bernstein:spectral:2} holds for all $f,g \in \Ran\EE_{-\Delta_\Omega}(\lambda)$ and all $m \in \NN_0$.
\end{prop}

\begin{proof}
	(i).~%
	By translation, we may assume that each semibounded coordinate $(a_j,b_j)$ is of the form $(0,\infty)$ or $(-\infty,0)$. We may
	then reflect this coordinate to the whole axis $\RR$ by means of Lemma~\ref{lemma:singeReflection}. In light of the above
	considerations, this reduces generalised rectangles to the three particular cases of $\RR^d$, hyperrectangles, and Cartesian
	products of these two. The latter are covered by	Lemma~\ref{lemma:bernstein-cartesian-product} as soon as the first two cases
	are established.

	For $\RR^d$, we may apply Lemma~\ref{lem:Bernstein-Criterion} with $\cM$ being the Schwartz space; cf.~also
	Remark~\ref{rmk:Bernstein}. It only remains to recall that by Fourier transform each $\Ran\EE_{-\Delta_{\RR^d}}(\lambda)$ is
	contained in the Schwartz space.

	Hyperrectangles can be reduced to the case of one dimensional bounded intervals $(a,b)$ by
	Lemma~\ref{lemma:bernstein-cartesian-product}. Here, all eigenfunctions of the negative second derivative on $(a,b)$ with
	Dirichlet or Neumann boundary conditions are finite trigonometric sums and, thus, clearly belong to $W^{\infty,2}((a,b))$. By
	(sesqui-)linearity, it now suffices to observe	that the one dimensional variant of~\eqref{eq:boundary-integration} holds for all
	eigenfunctions; in fact, for each choice of eigenfunctions $f,g$ we clearly have
	\begin{equation*}
		f^{(m)}(x)g^{(m+1)}(x) = 0, \quad x \in \{a,b\},\quad m\in\NN_0.
	\end{equation*}		
	This allows us to apply Lemma~\ref{lem:Bernstein-Criterion} for the respective spectral subspaces, which	concludes the proof of
	case~(i).

	(ii).~%
	The Dirichlet and Neumann eigenfunctions for the Laplacian on an equilateral triangle are explicitly known to be certain
	trigonometric functions, see, e.g.,~\cite{P98}. In particular, they belong to $W^{\infty,2}(\Omega)$, have a smooth extension to
	the whole plane $\RR^2$, and $-\Delta_\Omega$ acts on them as the differential expression $-\Delta$. In view of
	Lemma~\ref{lem:Bernstein-Criterion} and Remark~\ref{rmk:Bernstein}, it therefore suffices to show that
	\begin{equation*}
		\sum_{\abs{\alpha}=m} \frac{1}{\alpha!} (\partial^\alpha f) (\partial_\nu \partial^\alpha g) = 0
		\quad\text{ a.e.~on } \partial\Omega
	\end{equation*}
	for all eigenfunctions $f,g$ of $-\Delta_\Omega$ and all $m \in \NN$. This, in turn, holds by Lemma~\ref{lem:partial_der} if
	\begin{equation}\label{eq:Bernstein-equilateral}
		(\partial_{\nu^\perp}^{\alpha_1} \partial_\nu^{\alpha_2} f) \cdot (\partial_\nu \partial_{\nu^\perp}^{\alpha_1}
		\partial_\nu^{\alpha_2}g) = 0
		\quad\text{ a.e.~on } \partial\Omega
	\end{equation}
	for all $\alpha = (\alpha_1,\alpha_2) \in \NN_0^2$.
	
	Let $(\nu_1,\nu_2)$ be a fixed orthonormal basis of $\RR^2$. If $f$ is an eigenfunction with eigenvalue $p \geq 0$, the rotation
	symmetry of the Laplacian gives
	\begin{equation*}
		\partial_{\nu_2}^2 f
		=
		\Delta f - \partial_{\nu_1}^2 f
		=
		-p f - \partial_{\nu_1}^2 f
	\end{equation*}
	and, by induction,
	\begin{equation*}
		\partial_{\nu_2}^{2k} f
		=
		(-1)^k \sum_{j=0}^k \binom{k}{j} p^{k-j} \partial_{\nu_1}^{2j} f
	\end{equation*}
	for all $k \in \NN_0$.
	
	Suppose that $f$ satisfies Dirichlet boundary conditions. Since the outward unit normal $\nu$ on $\partial \Omega$ is locally
	constant, we then have $\partial_{\nu^\perp}^j f = 0$ almost everywhere on $\partial\Omega$ for all	$j \in \NN_0$ and, moreover,
	we may locally choose $(\nu_1,\nu_2) = (\nu^\perp,\nu)$.  If $\alpha_2 = 2k$ with $k \in \NN_0$ is even, this leads to
	\begin{equation*}
		\partial_{\nu^\perp}^{\alpha_1} \partial_\nu^{\alpha_2} f
		=
		(-1)^k \sum_{j=0}^k \binom{k}{j} p^{k-j} \partial_{\nu^\perp}^{2j+\alpha_1} f
		=
		0
		\quad\text{ a.e.~on }\partial\Omega
		.
	\end{equation*}
	If $\alpha_2$ is odd, then $\alpha_2+1$ is even, and the same argument applied with $g$ instead of $f$ yields
	$\partial_\nu \partial_{\nu^\perp}^{\alpha_1}\partial_\nu^{\alpha_2}g =
	\partial_{\nu^\perp}^{\alpha_1}\partial_\nu^{\alpha_2+1}g = 0$ almost everywhere on $\partial\Omega$.
	Thus,~\eqref{eq:Bernstein-equilateral} holds for all Dirichlet eigenfunctions $f,g$ of $-\Delta_\Omega$.

	If $f$ satisfies Neumann boundary conditions, we analogously have $\partial_{\nu^\perp}^j \partial_\nu f = 0$ almost	everywhere
	on $\partial\Omega$ for all $j \in \NN_0$, and~\eqref{eq:Bernstein-equilateral} follows in the same way as for Dirichlet
	eigenfunctions, but with the cases of $\alpha_2$ even and odd reversed.

	(iii).~%
	Both types of triangles are obtained by a single reflection from a square or an equilateral triangle, respectively, which
	are covered by~(i) and~(ii), so that case~(iii) follows by~\eqref{eq:Bernstein:spectral:singleReflection}.

	(iv).~%
	This follows immediately from the variant of~\eqref{eq:Bernstein:spectral:singleReflection} corresponding to the multiple
	reflections from Lemma~\ref{lemma:multipleReflection} and part~(i) applied for $\RR^2$.

	(v).~%
	Since sectors and generalised rectangles are the only unbounded domains under consideration and a Cartesian product of
	generalised rectangles is again a generalised rectangle, we may first treat Cartesian products with domains from~(i)--(iii) by
	Lemma~\ref{lemma:bernstein-cartesian-product}. Any further Cartesian products with sectors can then be handled by the variant
	of~\eqref{eq:Bernstein:spectral:singleReflection} corresponding to the multiple reflections from
	Lemma~\ref{lemma:multipleReflection:Cartesian}. This completes the proof.
\end{proof}%

\begin{rmk}\label{rmk:convexpolytopes}
  The proof for equilateral triangles in Proposition~\ref{prop:bernstein-pureLaplacian} can be considered as a more sophisticated
  version of the proof for bounded intervals in~(i). In fact, this proof can easily be adapted to handle also every bounded convex
  polytope in $\RR^d$ for which the corresponding eigenfunctions of the Laplacian belong to $W^{\infty,2}$. In this case, each
  eigenfunction can locally be extended by reflection along every boundary face (where the outward unit normal is constant). The
  (higher dimensional) variant of~\eqref{eq:Bernstein-equilateral} can then	be established with basically the same reasoning. This
  also gives an alternative proof for the case of hyperrectangles without reverting to
  Lemma~\ref{lemma:bernstein-cartesian-product}.
\end{rmk}

Let us now consider divergence-type operators $H = H_\Omega(A)$ as in Section~\ref{subsubsec:Divergence} with constant positive
definite symmetric matrix $A \in \RR^{d\times d}$. With $\tilde{\Omega} := A^{-1/2}(\Omega)$, define
$X \colon L^2(\Omega) \to L^2(\tilde{\Omega})$ by
\begin{equation*}
	Xf
	:=
	f \circ A^{1/2}
	.
\end{equation*}
It is then straightforward to verify that
\begin{equation*}
	X H_\Omega(A)
	\subset
	X (-\Delta_{\tilde{\Omega}})
\end{equation*}
via the corresponding relation~\eqref{eq:extRelationForm}, where the Laplacian $-\Delta_{\tilde{\Omega}}$ is equipped with the
same type of boundary conditions as $H_\Omega(A)$. Consequently, Lemma~\ref{lem:scaling} with $P = A^{1/2}$ yields a
Bernstein-type inequality for $H_\Omega(A)$ on all domains for which $\tilde{\Omega} = A^{-1/2}(\Omega)$ is one of the domains in
Proposition~\ref{prop:bernstein-pureLaplacian}. A reasonable class of such domains that does not depend on a specific choice of
$A$ is provided in the following result.

\begin{cor}\label{cor:bernstein-div-type}
	Let $A \in \RR^{d \times d}$ be symmetric and positive definite. If
	\begin{enumerate}[(i)]

		\item $\Omega = \RR^d$ or $\Omega$ is a half-space

		or

		\item $A$ is diagonal and $\Omega = \bigtimes_{j=1}^d (a_j,b_j)$ is a generalised rectangle,

	\end{enumerate}
	then $\Ran\EE_{H_\Omega(A)}(\lambda) \subset W^{\infty,2}(\Omega)$ for all $\lambda \geq 0$, and every
	$f \in \Ran\EE_{H_\Omega(A)}(\lambda)$ satisfies a Bernstein-type inequality with respect to
	$C_B(m,\lambda) = (\lambda / \sigma_{\min})^m$, where $\sigma_{\min}$ is the smallest eigenvalue of $A$.
\end{cor} 

\begin{proof}
	In case of~(i), $A^{-1/2}(\Omega)$ is again $\RR^d$ or a half-space, respectively, which can be handled for the Laplacian by
	Proposition~\ref{prop:bernstein-pureLaplacian}\,(i), upon possibly a global rotation of the coordinate system.
	
	In case of~(ii), $A^{-1/2}(\Omega)$ is again a generalised rectangle, which can likewise be handled for the Laplacian by
	Proposition~\ref{prop:bernstein-pureLaplacian}\,(i).
\end{proof}%

We finally discuss the harmonic oscillator $H_\Omega = -\Delta_\Omega + \abs{x}^2$ as in
Section~\ref{subsubsec:harmonicOscillator}. This is understood as the form sum of $-\Delta_\Omega$ and the densely defined
nonnegative closed form
\begin{equation*}
	\fv_\Omega [f , g]
	:=
	\langle \abs{\cdot}f , \abs{\cdot}g \rangle_{L^2(\Omega)}
\end{equation*}
defined on $\Dom(\fv_\Omega) = \{ f \in L^2(\Omega) \colon \abs{\cdot}f \in L^2(\Omega) \}$.

We consider the extension by single reflections: Let $M$ and $X \colon L^2(\Omega) \to L^2(\tilde{\Omega})$ be as in
Lemma~\ref{lemma:singeReflection}. We extend the form $\fv_\Omega$ by reflection to a densely defined nonnegative closed form
$\fv_{\tilde{\Omega}}$ on
$\Dom(\fv_{\tilde{\Omega}}) = \{ f \in L^2(\tilde{\Omega}) \colon f|_\Omega,(f\circ M)|_\Omega \in \Dom(\fv_\Omega) \}$ as
\begin{equation*}
	\fv_{\tilde{\Omega}}[f , g]
	=
	\fv[ f|_\Omega , g|_\Omega ] + \fv_\Omega[ (f\circ M)|_\Omega , (g\circ M)|_\Omega ]
	.
\end{equation*}
With the representation $X^*g = (g + \mu(g\circ M))|_\Omega$ for $g \in L^2(\tilde{\Omega})$ with $\mu \in \{ \pm 1\}$ as above,
see~\cite[Section~4]{ES19}, it is then easy to see that $X$ maps $\Dom(\fv_\Omega)$ into $\Dom(\fv_{\tilde{\Omega}})$, $X^*$ maps
$\Dom(\fv_{\tilde{\Omega}})$ into $\Dom(\fv_\Omega)$, and that
\begin{equation*}
	\fv_{\tilde{\Omega}} [ Xf , g ]
	=
	\fv_\Omega [ f , X^*g ]
	\qquad \forall\,f \in \Dom(\fv_\Omega),\ g \in \Dom(\fv_{\tilde{\Omega}})
	.
\end{equation*}
If now $M$ is a reflection symmetry of the mapping $\RR^d \ni x \mapsto \abs{x}$, and hence $\Dom(\fv_{\tilde{\Omega}})$ agrees
with $\{ f \in L^2(\tilde{\Omega}) \colon \abs{\cdot}f \in L^2(\tilde{\Omega}) \}$ and
$\fv_{\tilde{\Omega}}[ f , g] = \langle \abs{\cdot}f , \abs{\cdot}g \rangle_{L^2(\tilde{\Omega})}$ for
$f,g \in \Dom(\fv_{\tilde{\Omega}})$, this establishes relation~\eqref{eq:extRelationForm} for the harmonic oscillators on
$\Omega$ and $\tilde{\Omega}$, respectively; recall here that~\eqref{eq:extRelationForm} is already known for the pure Laplacians.
Thus, we have the extension relation $H_{\tilde{\Omega}}X \supset XH_\Omega$ in this situation, and we obtain the following
result from the Bernstein-type inequality for the harmonic oscillator on $\RR^d$.

\begin{prop}\label{prop:harmonicOscillator}
	Let $\Omega$ be one of the following domains:
	\begin{enumerate}[(i)]

		\item a generalised rectangle of the form $\bigtimes_{j=1}^d (a_j,b_j)$ with $a_j,b_j \in \{ 0,\pm\infty \}$, $a_j < b_j$;

		\item the sector $S_{\pi/4} \subset \RR^2$;

		\item a finite Cartesian product of the above.

	\end{enumerate}
	Then, $\Ran\EE_{H_\Omega}(\lambda) \subset W^{\infty,2}(\Omega)$ for all $\lambda \geq 0$, and, for every choice of
	$\delta > 0$, every function $f \in \Ran\EE_{H_\Omega}(\lambda)$ satisfies a Bernstein-type inequality with respect to
	$C_B(m,\lambda) = (2\delta)^{2m} \ee^{\ee/\delta^2} (m!)^2 \ee^{2\sqrt{\lambda}/\delta}$.
\end{prop}

\begin{proof}
	(i).~%
	The above considerations justify to	reflect	every semibounded coordinate $(a_j,b_j)$ of $\Omega$ to $\RR$ by means of
	Lemma~\ref{lemma:singeReflection}; cf.~the proof of Proposition~\ref{prop:bernstein-pureLaplacian}\,(i). This reduces the stated
	generalised rectangles to the case $\Omega = \RR^d$. The latter has essentially been discussed in Proposition~4.3\,(i)
	of~\cite{BJPS18}, which we have adapted to our purposes in Proposition \ref{prop:BJPS} below.

	(ii).~%
	We may use Lemma~\ref{lemma:singeReflection} to reflect the sector $S_{\pi/4}$ to the quadrant $(0,\infty)^2$, which can be
	handled by~(i).

	(iii).~%
	Since the Cartesian product of quadrants and generalised rectangles of the form in~(i) are again generalised rectangles
	of the same form, this follows analogously as in~(ii).
\end{proof}%

\section{Proof of Theorem~\ref{thm:LS}}\label{sec:main_proof}

Our main result, Theorem~\ref{thm:LS}, follows from a more general statement that can be formulated without an operator theoretic
framework and is at its core complex analytic in nature. This allows one to prove the result independently of the specific
operator and the chosen boundary conditions at hand. In this regard, we use the adapted notion for Bernstein-type inequalities
introduced in Section~\ref{sec:bernstein} here.

It is also convenient to drop in this section the parameters $\eta$ and $\varrho$ from the notion of coverings: For a domain
$\Omega \subset \RR^d$ and $\kappa \ge 1$, $l = (l_1,\dots,l_d) \in (0,\infty)^d$, we call a family $\{ Q_j \}_{j\in J}$ of
non-empty bounded convex open subsets $Q_j \subset \Omega$, with $J$ a finite or countably infinite index set,
a~\emph{$(\kappa,l)$-covering of $\Omega$} if
\begin{enumerate}
 \renewcommand{\theenumi}{\roman{enumi}}
 \item $\Omega \setminus \bigcup_{j\in J} Q_j$ has Lebesgue measure zero;\label{it:2:covering}
 \item $\sum_{j\in J} \norm{g}{L^2(Q_j)}^2 \le \kappa \norm{g}{L^2(\Omega)}^2$ for all $g \in L^2(\Omega)$;\label{it:2:overlap}
 \item each $Q_j$ lies in a hyperrectangle with sides of length $l$ parallel to coordinate axes.\label{it:2:supercube}
\end{enumerate}
Again, the quantity $\kappa$ in~\eqref{it:2:overlap} can be interpreted as the maximal essential overlap between the sets $Q_j$.

We now state the general result that is at the core of Theorem~\ref{thm:LS}.
\begin{prop}\label{prop:LS-abstract-version}
 Let $\Omega \subset \RR^d$ be a domain, $\{Q_j\}_{j \in J}$ a $(\kappa,l)$-covering of $\Omega$, and let
 $f \in W^{\infty,2}(\Omega)$ satisfy a Bernstein-type inequality with respect to $C_B \colon \NN_0 \to (0,\infty)$ such that
 \begin{equation}
  \label{eq:BernsteinSum}
  h := \sum_{m \in \NN_0} C_B(m)^{1/2} \frac{(10\norm{l}{1})^m}{m!} < \infty.
 \end{equation}
 Then, for every measurable subset $\omega \subset \Omega$ and every choice of linear bijections $\Psi_j \colon \RR^d \to \RR^d$,
 $j \in J$, with $\nu := \inf_{j \in J} \abs{\bijection_j(Q_j \cap \omega)} / \diam(\bijection_j(Q_j))^d > 0$ we have
 \begin{equation}\label{eq:LS-abstract-version}
  \norm{f}{L^2(\Omega)}^2
  \le
  \frac{\kappa}{6}\Bigl( \frac{24d\tau_d}{\nu} \Bigr)^{2\frac{\log\kappa}{\log 2} + 4\frac{\log h}{\log 2} + 5}
    \norm{f}{L^2(\omega)}^2.
 \end{equation}
\end{prop}

Note that the condition $\nu > 0$ in the above result necessarily requires $\omega$ to have positive measure.

Before we get to the proof of Proposition~\ref{prop:LS-abstract-version}, we explain how this leads to Theorem~\ref{thm:LS}.

\begin{proof}[Proof of Theorem~\ref{thm:LS}]
 It is clear that the hypotheses of Proposition~\ref{prop:LS-abstract-version} are satisfied for each $f \in \Ran\EE_H(\lambda)$
 with $C_B(m) = C_B(m, \lambda)$ and $h = h(\lambda)$. Moreover, choosing the bijective linear mapping
 $\bijection_j \colon \RR^d \to \RR^d$ according to the $(\kappa,\rho,l,\eta)$-covering such that
 $\abs{\bijection_j(Q_j)} \ge \eta \diam(\bijection_j(Q_j))^d$, for each $j \in J$ we have
 \begin{equation*}
  \frac{\abs{\bijection_j(Q_j\cap\omega)}}{\diam(\bijection_j(Q_j))^d}
  =
  \frac{\abs{Q_j \cap \omega}}{\abs{Q_j}} \cdot \frac{\abs{\bijection_j(Q_j)}}{\diam(\bijection_j(Q_j))^d}
  \ge
  \frac{\gamma\rho^d}{l_1\cdot\dots\cdot l_d} \cdot \eta
  ,
 \end{equation*}
 where we used that $Q_j$ is contained in a hyperrectangle of the form $y + \bigtimes_{j=1}^d (0,l_j)$ and that $\omega$ is
 $(\gamma,\rho)$-thick in $\Omega$ and each $Q_j$ contains a hypercube of the form $x + (0,\rho)^d$. This provides a corresponding
 lower bound on $\nu$. The inequality~\eqref{eq:LS} is now a direct consequence of~\eqref{eq:LS-abstract-version}.
\end{proof}%

The rest of this section is devoted to the proof of Proposition~\ref{prop:LS-abstract-version}. This proof is essentially based
on~\cite{EV20} and~\cite{Egi21}, which are in turn inspired by~\cite{Kovrijkine:01,Kovrijkine:thesis};
cf.~also~\cite{WWZZ19,BJPS18}. We give a brief outline in advance:

The Bernstein-type inequality together with condition~\eqref{eq:BernsteinSum} ensures that the given function $f$ is analytic in
$\Omega$, see Lemma~\ref{lem:Bernstein} below. On some of the $Q_j$ of a specific type, called~\emph{good} elements of the
covering, we can even extend $f$ to an analytic function on a sufficiently large complex neighbourhood of $Q_j$, which allows to
employ methods from complex analysis and obtain a local estimate related to~\eqref{eq:LS-abstract-version} on each of such good
$Q_j$, see Lemma~\ref{lem:local_estimate} below. It turns out that the local estimates can be made uniform over the class of good
$Q_j$ and that the contribution of this class is so large that these local estimates can finally be summed up to the desired
global estimate~\eqref{eq:LS-abstract-version}, see Subsection~\ref{subsec:good_and_bad} below.

\subsection{Analyticity}\label{subsec:analyticity}

We first show that a Bernstein-type inequality together with condition~\eqref{eq:BernsteinSum} already implies analyticity of the
function under consideration. The precise statement reads as follows.

\begin{lemma}\label{lem:Bernstein}
 Let $\Omega \subset \RR^d$ be open, and let $f \in W^{\infty,2}(\Omega)$ satisfy a Bernstein-type inequality with respect to
 $C_B \colon \NN_0 \to (0,\infty)$ such that
 \begin{equation}\label{eq:analytic}
  \sum_{m = 0}^\infty C_B(m)^{1/2} \frac{r^m}{m!} < \infty
 \end{equation}
 for some $r > 0$. Then, $f$ is analytic in $\Omega$, that is, $f$ can locally be expanded into a convergent power series.
\end{lemma}

\begin{proof}
 Without loss of generality, we may assume that $r < 1$. Let $y \in \Omega$, and let $B = B_\epsilon(y)$ be an open ball around $y$
 with $\epsilon  < r/d$ such that its closure is contained in $\Omega$. Since $B$ satisfies the cone condition, by Sobolev
 embedding there exists a constant $c_d > 0$, depending only on the dimension, such that
 \begin{equation*}
  \norm{g}{L^\infty(B)}
  \le
  c_d \norm{g}{W^{d,2}(B)}
  \quad\text{ for all }\quad
  g \in W^{d,2}(B),
 \end{equation*}
 see, e.g.,~\cite[Theorem~4.12]{AF03}. Applying this to $g = \partial^\alpha f|_B$ with $\abs{\alpha} = m \in \NN_0$, we obtain
 \begin{align*}
  \norm{ \partial^\alpha f }{L^\infty(B)}^2
  &\le
  c_d^2 \norm{\partial^\alpha f}{W^{d,2}(B)}^2 \le c_d^2 \norm{\partial^\alpha f}{W^{d,2}(\Omega)}^2\\
  &=
  c_d^2 \sum_{\abs{\beta} \le d} \norm{\partial^{\alpha+\beta}f}{L^2(\Omega)}^2
   \le
   c_d^2 \sum_{k=m}^{m+d} \sum_{\abs{\beta}=k} \norm{\partial^\beta f}{L^2(\Omega)}^2\\   
  &\le
  c_d^2 \sum_{k=m}^{m+d} \sum_{\abs{\beta}=k} \frac{k!}{\beta!} \norm{\partial^\beta f}{L^2(\Omega)}^2
  \le
  c_d^2 \norm{f}{L^2(\Omega)}^2 \sum_{k=m}^{m+d} C_B(k).
 \end{align*}
 Taking into account that $r < 1$, we have
 \begin{equation*}
  \sum_{k=m}^{m+d} C_B(k)^{1/2} \frac{r^k}{k!}
  \ge
  \frac{r^{m+d}}{(m+d)!} \sum_{k=m}^{m+d} C_B(k)^{1/2},
 \end{equation*}
 so that the above yields
 \begin{align*}
  \norm{ \partial^\alpha f }{L^\infty(B)}
  &\le
  c_d \norm{f}{L^2(\Omega)} \biggl( \sum_{k=m}^{m+d} C_B(k) \biggr)^{1/2}
   \le
   c_d \norm{f}{L^2(\Omega)} \sum_{k=m}^{m+d} C_B(k)^{1/2}\\
  &\le
  c_d \norm{f}{L^2(\Omega)} \frac{(m+d)!}{r^{m+d}} \sum_{k=m}^{m+d} C_B(k)^{1/2} \frac{r^k}{k!}\\
  &\le
  c_d \norm{f}{L^2(\Omega)} \frac{(m+d)!}{r^{m+d}} \sum_{k = 0}^\infty C_B(k)^{1/2} \frac{r^k}{k!}
  .
 \end{align*}
 In view of~\eqref{eq:analytic}, this is sufficient to conclude that the Taylor series of $f$ converges absolutely in
 $B_\epsilon(y)$ and agrees with $f$ there, see also~\cite[Proposition~2.2.10]{KP02}. Hence, $f$ is analytic in $\Omega$.
\end{proof}%

\begin{rmk}
 \label{rmk:identity_theorem}
 As in the complex analytic case, an analytic function $f \colon Q \to \CC$ on a domain $Q \subset \RR^d$ that vanishes on a
 non-empty open subset $U \subset Q$ must vanish on the whole of $Q$. The proof is essentially the same as for complex analytic
 functions, cf., e.g.,~\cite[Conclusion~1.2.12]{Schei05}.
\end{rmk}

\subsection{The local estimate}\label{subsec:local_estimate}

For $r > 0$ denote $D(r) := \{ z \in \CC \colon \abs{z} < r \}$. 
The lemma below is an adjusted version of~\cite[Lemma 1]{Kovrijkine:01}, where it is presented for functions that are analytic in
$D(0,5)$. Indeed, its proof shows that it suffices to have analyticity in a complex disk of radius strictly larger
than $4$.

\begin{lemma}[{\cite[Lemma~1]{Kovrijkine:01}}]\label{lem:kovrijkine_orig}
 Let $\phi \colon D(4+\epsilon) \to \CC$ for some $\epsilon > 0$ be an analytic function with $\abs{\phi(0)}\geq 1$. Moreover, let
 $E \subset [0,1]$ be measurable with positive measure. Then,
 \begin{equation*}
  \sup_{t\in [0,1]} \abs{\phi(t)}
  \leq
  \left( \frac{12}{\abs{E}} \right)^{2\frac{\log M_{\phi}}{\log 2}} \sup_{t\in E} \abs{\phi(t)},
 \end{equation*}
 where $M_{\phi} = \sup_{z\in D(4)} \abs{\phi(z)}$.
\end{lemma}

For $r = (r_1,\dots,r_d) \in (0,\infty)^d$ let $D_r := D(r_1) \times \dots \times D(r_d) \subset \CC^d$. The above lemma can now
be combined with a dimension reduction argument to obtain the following local estimate, which is inspired
by~\cite{Kovrijkine:01}. Its formulation only requires that the considered function $f$ has a complex analytic extension to a
sufficiently large complex neighbourhood of $Q = Q_j$. A similar statement is implicitly contained
in~\cite[Section~5]{EV20},~\cite[Section 3.3.3]{BJPS18}, and~\cite[Proof of Lemma~2.1]{WWZZ19}.

\begin{lemma}\label{lem:local_estimate}
 Let $Q \subset \RR^d$ be a non-empty bounded convex open set contained in a hyperrectangle with sides of length
 $l = (l_1,\dots,l_d) \in (0,\infty)^d$ parallel to coordinate axes. Moreover, let $f \colon Q \to \CC$ be a non-vanishing
 function having an analytic extension $F \colon Q + D_{4l} \to \CC$ with bounded modulus.
 
 Then, for every measurable set
 $\omega \subset \RR^d$ and every linear bijection $\bijection \colon \RR^d \to \RR^d$ we have
 \begin{align}
  \norm{f}{L^2(Q \cap \omega)}^2
  &\ge
  \frac{1}{2} \Bigl( \frac{\abs{\bijection(Q \cap \omega)}}{24d\tau_d\diam(\bijection(Q))^d} \Bigr)^{4\frac{\log M}{\log 2}}
   \cdot
   \frac{\abs{Q \cap \omega}}{\abs{Q}}
   \cdot
   \norm{f}{L^2(Q)}^2\label{eq:local_estimate:1}\\
  &\ge
  12 \Bigl( \frac{\abs{\bijection(Q \cap \omega)}}{24d\tau_d\diam(\bijection(Q))^d} \Bigr)^{4\frac{\log M}{\log 2}+1}
   \norm{f}{L^2(Q)}^2
  \label{eq:local_estimate:2}
 \end{align}
 with
 \begin{equation*}
  M := \frac{\sqrt{\abs{Q}}}{\norm{f}{L^2(Q)}} \cdot \sup_{z \in Q + D_{4l}} \abs{F(z)} \ge 1.
 \end{equation*}
\end{lemma}
 
\begin{proof}
 Consider the open set 
 \begin{equation*}
  W
  :=
  \biggl\{ x \in Q \colon \abs{f(x)}
   < \Bigl( \frac{\abs{\bijection(Q \cap \omega)}}{24d\tau_d\diam(\bijection(Q))^d} \Bigr)^{2\frac{\log M}{\log 2}} \cdot
   \frac{\norm{f}{L^2(Q)}}{\sqrt{\abs{Q}}} \biggr\}.
 \end{equation*}
 In order to prove~\eqref{eq:local_estimate:1}, it suffices to show the inequality $\abs{Q \cap \omega} \ge 2\abs{W}$ since then
 $\abs{(Q \cap \omega) \setminus W} \ge \abs{Q \cap \omega} / 2$ and, thus, by definition of $W$,
 \begin{equation*}
  \norm{f}{L^2(Q \cap \omega)}^2
  \ge
  \norm{f}{L^2((Q \cap \omega) \setminus W)}^2
  \ge
  \frac{\abs{Q \cap \omega}}{2}
   \cdot \Bigl( \frac{\abs{\bijection(Q \cap \omega)}}{24d\tau_d\diam(\bijection(Q))^d} \Bigr)^{4\frac{\log M}{\log 2}}
   \cdot \frac{\norm{f}{L^2(Q)}^2}{\abs{Q}},
 \end{equation*}
 which agrees with the claim.

 In order to show $\abs{Q \cap \omega} \ge 2\abs{W}$, we may suppose that $W \neq \emptyset$. Choose a point $y \in Q$ with
 $\abs{f(y)} \ge \norm{f}{L^2(Q)} / \sqrt{\abs{Q}}$. We then claim that there is a line segment $I = I(y,W,Q) \subset Q$ starting
 at $y$ such that
 \begin{equation}\label{eq:dimension_reduc}
  \frac{\abs{I \cap W}}{\abs{I}}
  \ge
  \frac{\abs{\bijection(W)}}{d\tau_d\diam(\bijection(Q))^d}
  .
 \end{equation}
 Indeed, using spherical coordinates around $\bijection(y)$, we write
 \begin{equation*}
  \abs{\bijection(W)}
  =
  \int_{\abs{\xi} = 1} \int_0^\infty \chi_{\bijection(W)}(\bijection(y) + s\xi)s^{d-1} \dd s\dd\sigma(\xi),
 \end{equation*}
 so that
 \begin{equation*}
  \abs{\bijection(W)}
  \le
  d\tau_d \int_0^\infty \chi_{\bijection(W)}(\bijection(y) + s\zeta)s^{d-1} \dd s
 \end{equation*}
 for some $\zeta\in\RR^d$ with $\abs{\zeta}=1$; recall that $d\tau_d$ agrees with the $(d-1)$-dimensional volume of the unit sphere
 in $\RR^d$. Since $Q$ is convex by hypothesis, the intersection $I := \{ y + s\bijection^{-1}(\zeta) \colon s \ge 0 \} \cap Q$ is
 then a line segment starting at $y$ satisfying
 \begin{equation*}
  \abs{\bijection(W)}
  \le
  d\tau_d \abs{\bijection(I)}^{d-1} \int_0^\infty \chi_{\bijection(I \cap W)}(\bijection(y) + s\zeta) \dd s
  =
  d\tau_d \abs{\bijection(I)}^{d-1} \abs{\bijection(I \cap W)}
  .
 \end{equation*}
 Taking into account that $\abs{\bijection(I)}^d \le \diam(\bijection(Q))^d$ and in view of the identity
 $\abs{I\cap W}/\abs{I} = \abs{\bijection(I\cap W)}/\abs{\bijection(I)}$, this proves~\eqref{eq:dimension_reduc}.

 Since the set $Q$ is open and $y \in Q$ and setting
 $\tilde\zeta = \bijection^{-1}(\zeta) / \abs{\bijection^{-1}(\zeta)} \in \RR^d$, we have $\abs{I}\tilde{\zeta} z \in D_{4l}$ for
 all $z \in D(4+\epsilon)$ for some sufficiently small $\epsilon > 0$.
 The mapping
 \begin{equation*}
  z \mapsto \phi(z) := \frac{\sqrt{\abs{Q}}}{\norm{f}{L^2(Q)}} \cdot F(y + \abs{I}\tilde\zeta z)
 \end{equation*}
 therefore defines an analytic function $\phi \colon D(4 + \epsilon) \to \CC$. Moreover, we have
 $\sup_{t\in[0,1]}\abs{\phi(t)} \ge \abs{\phi(0)} = \sqrt{\abs{Q}}\abs{f(y)}/\norm{f}{L^2(\Omega)} \ge 1$ by the choice of $y$.
 Hence, with
 \begin{equation*}
  M_\phi
  :=
  \sup_{z \in D(4)} \abs{\phi(z)}
  \le
  \frac{\sqrt{\abs{Q}}}{\norm{f}{L^2(Q)}} \cdot \sup_{z \in y + D_{4l}} \abs{F(z)}
  \le
  M,
 \end{equation*}
 applying Lemma~\ref{lem:kovrijkine_orig} to $\phi$ and
 $E := \{ t \in [0,1] \colon y + \abs{I}\tilde \zeta t \in I \cap W \} \subset [0,1]$ gives
 \begin{equation*}
  \sup_{t \in E} \abs{\phi(t)}
  \ge
  \Bigl( \frac{\abs{E}}{12} \Bigr)^{2\frac{\log M_\phi}{\log 2}} \sup_{t\in[0,1]} \abs{\phi(t)}
  \ge
  \Bigl( \frac{\abs{E}}{12} \Bigr)^{2\frac{\log M}{\log 2}}.
 \end{equation*}
 Inserting the definition of $\phi$ and recalling that $F|_Q = f$, this can be rewritten as
 \begin{equation*}
  \sup_{t \in E} \abs{f(y + \abs{I}\tilde\zeta t)}
  \ge
  \Bigl( \frac{\abs{E}}{12} \Bigr)^{2\frac{\log M}{\log 2}} \cdot \frac{\norm{f}{L^2(Q)}}{\sqrt{\abs{Q}}}.
 \end{equation*}
 In light of $\sup_{x \in W} \abs{f(x)} \ge \sup_{x \in I \cap W} \abs{f(x)} = \sup_{t \in E} \abs{f(y + \abs{I}\tilde\zeta t)}$
 and the identity $\abs{E} = \abs{I \cap W} / \abs{I}$, we conclude from~\eqref{eq:dimension_reduc} and the latter that
 \begin{equation*}
  \sup_{x \in W} \abs{f(x)}
  \ge
  \Bigl( \frac{\abs{\bijection(W)}}{12d\tau_d\diam(\bijection(Q))^d} \Bigr)^{2\frac{\log M}{\log 2}}
   \cdot \frac{\norm{f}{L^2(Q)}}{\sqrt{\abs{Q}}}
  .
 \end{equation*}
 Combining the above with the definition of $W$ and noting that
 $\frac{\abs{\Psi(Q\cap \omega)}}{\abs{\Psi(W)}} = \frac{\abs{Q\cap \omega}}{\abs{W}}$, we obtain 
 \begin{equation*}
  \sup_{x \in W} \abs{f(x)}
  \le
  \Bigl( \frac{\abs{\bijection(Q \cap \omega)}}{24d\tau_d\diam(\bijection(Q))^d} \Bigr)^{2\frac{\log M}{\log 2}}
   \cdot \frac{\norm{f}{L^2(\Omega)}}{\sqrt{\abs{Q}}}
  \le
  \Bigl( \frac{\abs{Q \cap \omega}}{2\abs{W}} \Bigr)^{2\frac{\log M}{\log 2}} \sup_{x \in W} \abs{f(x)}
  ,
 \end{equation*}
 which requires $\abs{Q \cap \omega} \ge 2\abs{W}$ since $f$, as a non-zero analytic function on the domain $Q$, cannot vanish on
 the whole open set $W$, see Remark~\ref{rmk:identity_theorem}. This completes the proof of~\eqref{eq:local_estimate:1}.
 
 Finally, we have $\abs{\bijection(Q)} \le d\tau_d\diam(\bijection(Q))^d$, which can be seen, for instance, with the same argument
 as used to derive~\eqref{eq:dimension_reduc}, with $\bijection(Q)$ instead of $\bijection(W)$. This
 shows~\eqref{eq:local_estimate:2} and, thus, concludes the proof of the lemma.
\end{proof}%

\begin{rmk}\label{rmk:bijection}
 The bijection $\bijection$ in the above lemma may be used to compensate for `unfavourable' geometry, that is, for cases where the
 ratio $\abs{Q} / \diam(Q)^d$ is small. More precisely, we may choose $\bijection$ such that the second factor on the right-hand
 side of
 \begin{equation*}
   \frac{\abs{\bijection(Q \cap \omega)}}{\diam(\bijection(Q))^d}
   =
   \frac{\abs{Q\cap\omega}}{\abs{Q}} \cdot \frac{\abs{\bijection(Q)}}{\diam(\bijection(Q))^d}
 \end{equation*}
 is maximal. For instance, for the rectangle $Q = (0,l)\times (0,1)\subset\RR^2$, $l>0$, we have
 $\abs{Q} / \diam(Q)^2 = l / (1+l^2)$, and the latter can be arbitrarily small with large $l$. On the other hand, choosing the
 bijection $\bijection \colon \RR^2 \to \RR^2$ as $\bijection(x_1,x_2) = (x_1/l,x_2)$, we have $\bijection(Q) = (0,1)^2$ and,
 thus, $\abs{\bijection(Q)} / \diam(\bijection(Q))^2 = 1/2$.
\end{rmk}

\subsection{Good and bad elements of the covering}\label{subsec:good_and_bad}

In the context of Proposition~\ref{prop:LS-abstract-version}, the local estimate from Lemma~\ref{lem:local_estimate} raises two
immediate questions: Firstly, when can $f$ (or rather its restriction $f|_{Q_j}$) be extended to an analytic function on a
sufficiently large complex neighbourhood of $Q_j$? Secondly, can the corresponding quantity $M$ be bounded from above uniformly
over the $Q_j$?

The analyticity of $f$ allows to deal with these two questions in terms of convergence properties of suitable Taylor expansions.
We are in fact able to handle these for a sufficiently large class of covering elements $Q_j$. Recall that $\{Q_j\}_{j\in J}$ is a
$(\kappa,l)$-covering of $\Omega$ and that $f \colon \Omega \to \CC$ satisfies a Bernstein-type inequality with respect to
$C_B \colon \NN_0 \to (0,\infty)$ such that~\eqref{eq:BernsteinSum} holds. Now, we say that $Q_j$ is a \emph{good} element of the
covering if
\begin{equation}\label{eq:good}
 \sum_{\abs{\alpha} = m} \frac{1}{\alpha!}\norm{\partial^\alpha f}{L^2(Q_j)}^2
 \le
 2^{m+1}\kappa \frac{C_B(m)}{m!} \norm{f}{L^2(Q_j)}^2
 \quad\text{ for all }\quad
 m \in \NN,
\end{equation}
and we call it \emph{bad} otherwise; note that~\eqref{eq:good} is automatically satisfied for $m = 0$ for~\emph{all} $Q_j$ since
necessarily $C_B(0) \ge 1$. We claim that
\begin{equation}\label{eq:bad-mass}
 \sum_{j \colon Q_j\text{ good}} \norm{f}{L^2(Q_j)}^2
 \geq
 \frac{1}{2}\norm{f}{L^2(\Omega)}^2.
\end{equation}
Indeed, using the definition of bad elements, property~\eqref{it:overlap} of the $(\kappa,l)$-covering, and the Bernstein-type
inequality for $f$, we obtain
\begin{align*}
 \sum_{j \colon Q_j\text{ bad}} \norm{f}{L^2(Q_j)}^2
 &\leq
 \sum_{j \colon Q_j\text{ bad}}\sum_{m\in\NN}\sum_{\abs{\alpha}=m} \frac{m! 2^{-m-1}}{\alpha! \kappa C_B(m)}
  \norm{\partial^\alpha f}{L^2(Q_j)}^2\\
 &\leq
 \sum_{j\in J} \sum_{m\in\NN} \sum_{\abs{\alpha}=m} \frac{m! 2^{-m-1}}{\alpha! \kappa C_B(m)}
  \norm{\partial^\alpha f}{L^2(Q_j)}^2\\
 &\leq
 \sum_{m\in\NN}\sum_{\abs{\alpha}=m} \frac{m! 2^{-m-1}}{\alpha! C_B(m)} \norm{\partial^\alpha f}{L^2(\Omega)}^2\\
 &\leq
 \norm{f}{L^2(\Omega)}^2\sum_{m\in\NN} 2^{-m-1}\\
 &=
 \frac{1}{2}\norm{f}{L^2(\Omega)}^2,
\end{align*}
proving~\eqref{eq:bad-mass}. Thus, the class of good elements is large enough that the contribution of the bad ones can be
subsumed in the contribution of the good ones, see the proof of Proposition~\ref{prop:LS-abstract-version} below; in particular,
good elements exist. 
 
Addressing the convergence properties of suitable Taylor expansions of $f$ on good elements of the covering requires a certain
control on corresponding derivatives of $f$. In fact, in every good element $Q_j$ there exists a point $x_0$ such that 
\begin{equation*}
 \sum_{\abs{\alpha}=m} \frac{1}{\alpha!} \abs{\partial^\alpha f(x_0)}^2
 \le
 4^{m+1} \kappa \frac{C_B(m)}{m!} \cdot \frac{\norm{f}{L^2(Q_j)}^2}{\abs{Q_j}}
 \quad\text{ for all }\quad
 m \in \NN_0.
\end{equation*}
In order to see this, we assume by contradiction that for all $x\in Q_j$ there exists $m=m(x)$ such that
\[ 
 \sum_{\abs{\alpha}=m}\frac{1}{\alpha!}\abs{\partial^\alpha f(x)}^2
 >
 \frac{4^{m+1}\kappa C_B(m)}{m! \abs{Q_j}}\norm{f}{L^2(Q_j)}^2.
\]
We estimate further by the sum over all $m\in\NN_0$ in order to get rid of the $x$-dependence and obtain
\[
 \sum_{m\in\NN_0} \sum_{\abs{\alpha}=m}\frac{m! 4^{-m-1}}{\alpha! \kappa C_B(m)} \abs{\partial^\alpha f(x)}^2
 >
 \frac{1}{\abs{Q_j}}\norm{f}{L^2(Q_j)}^2
\]
for all $x \in Q_j$. Integration over $Q_j$ and the definition of good elements yield 
\begin{align*}
 \norm{f}{L^2(Q_j)}^2
 &<
 \sum_{m\in\NN_0} \frac{m! 4^{-m-1}}{\kappa C_B(m)}\sum_{\abs{\alpha}=m} \frac{1}{\alpha!} \norm{\partial^\alpha f}{L^2(Q_j)}^2\\
 &\leq
 \norm{f}{L^2(Q_j)}^2 \sum_{m\in\NN_0} \frac{1}{2^{m+1}} = \norm{f}{L^2(Q_j)}^2
 ,
\end{align*}
leading to a contradiction. For some $x_0 \in Q_j$ we thus have, in particular,
\begin{equation}\label{eq:existence-point}
 \abs{\partial^\alpha f(x_0)}^2
 \le
 \Bigl( \frac{m!}{\alpha!} \Bigr)\abs{\partial^\alpha f(x_0)}^2
 \le
 4^{m+1} \kappa C_B(m) \frac{\norm{f}{L^2(Q_j)}^2}{\abs{Q_j}}
\end{equation}
for all $m \in \NN_0$ and all $\alpha \in \NN_0^d$ with $\abs{\alpha} = m$.

We are now in position to conclude the proof of Proposition~\ref{prop:LS-abstract-version}.

\begin{proof}[Proof of Proposition~\ref{prop:LS-abstract-version}]
 In light of~\eqref{eq:BernsteinSum} and Lemma~\ref{lem:Bernstein}, $f$ is analytic in $\Omega$. Moreover, we may assume that $f$
 does not vanish on $\Omega$ and, therefore, also on none of the $Q_j$ by Remark~\ref{rmk:identity_theorem}.
 Let $Q_j$ be a good element of the covering as defined in~\eqref{eq:good}, and let $x_0\in Q_j$ be a point as
 in~\eqref{eq:existence-point}. For every $z \in x_0 + D_{5l}$ we then have
 \begin{align*}
  \sum_{\alpha \in \NN_0^d}\frac{\abs{\partial^\alpha f(x_0)}}{\alpha!}\abs{(z-x_0)^\alpha}
  &\leq
  \sum_{m\in \NN_0}\sum_{\abs{\alpha}=m}\frac{1}{\alpha!} (\kappa C_B(m))^{1/2} 2^{m+1} (5l)^\alpha
   \frac{\norm{f}{L^2(Q_j)}}{\sqrt{\abs{Q_j}}}\\
  &=
  2 \kappa^{1/2}\frac{\norm{f}{L^2(Q_j)}}{\sqrt{\abs{Q_j}}} \sum_{m\in\NN_0} C_B(m)^{1/2} \frac{(10 \norm{l}{1})^m}{m!}\\
  &=
  2 \kappa^{1/2}\frac{\norm{f}{L^2(Q_j)}}{\sqrt{\abs{Q_j}}} h < \infty,
 \end{align*}
 with $h$ as in \eqref{eq:BernsteinSum}; note that at the second step we used
 $\sum_{\abs{\alpha} = m} l^\alpha/\alpha! = \norm{l}{1}^m / m!$. Hence, the Taylor
 expansion of $f$ around $x_0$ converges in the complex polydisk $x_0 + D_{5l}$. Now, since $x_0 \in Q_j$ and $Q_j$ is open and
 contained in a hypercube with sides of length $l$ parallel to coordinate axes, for some sufficiently small $\epsilon > 0$ we
 have
 \begin{equation*}
  Q_j + D_{4l+\epsilon} \subset x_0 + D_{5l}.
 \end{equation*}
 Taking into account Remark~\ref{rmk:identity_theorem}, the Taylor series of $f$ around the point $x_0$ defines 
 therefore an analytic extension $F \colon Q_j + D_{4l+\epsilon} \to \CC$ of $f|_{Q_j}$ with
 \begin{equation*}
  M_j := \frac{\sqrt{\abs{Q_j}}}{\norm{f}{L^2(Q_j)}} \cdot \sup_{z \in Q_j + D_{4l}} \abs{F(z)}
  \le
  2 \kappa^{1/2}h
  =:
  M.
 \end{equation*}
 Together with the definition of $\nu$, it now follows from Lemma~\ref{lem:local_estimate} that
 \begin{equation}\label{eq:local-estimate}
  \begin{aligned}
   \norm{f}{L^2(Q_j \cap \omega)}^2
   &\ge
   12 \Bigl( \frac{\abs{\bijection_j(Q_j \cap \omega)}}{24d\tau_d\diam(\bijection_j(Q_j))^d} \Bigr)^{4\frac{\log M_j}{\log 2}+1}
    \cdot \norm{f}{L^2(Q_j)}^2\\
   &\ge
   12 \Bigl( \frac{\nu}{24d\tau_d} \Bigr)^{4\frac{\log M}{\log 2}+1} \cdot \norm{f}{L^2(Q_j)}^2.
  \end{aligned}
 \end{equation}
 Using~\eqref{eq:bad-mass} and property~\eqref{it:overlap} of the covering, summing~\eqref{eq:local-estimate} over all good
 elements yields
 \begin{align*}
  \norm{f}{L^2(\omega)}^2
  &\geq
  \frac{1}{\kappa}\sum_{j\in J}\norm{f}{L^2(Q_j\cap \omega)}^2
   \geq \frac{1}{\kappa}\sum_{j \colon Q_j\text{ good}}\norm{f}{L^2(Q_j\cap \omega)}^2\\
  &\geq
  \frac{12}{\kappa}\left(\frac{\nu}{24d \tau_d}\right)^{4\frac{\log M}{\log 2}+1}
   \sum_{j \colon Q_j\text{ good}}\norm{f}{L^2(Q_j)}^2\\
  &\geq
  \frac{6}{\kappa}\left(\frac{\nu}{24d \tau_d}\right)^{4\frac{\log M}{\log 2}+1}\norm{f}{L^2(\Omega)}^2.
 \end{align*}
 In order to complete the proof, it only remains to observe that
 \begin{equation*}
  4\frac{\log M}{\log 2}+1
  =
  2\frac{\log \kappa}{\log 2} + 4\frac{\log h}{\log 2} + 5.\qedhere
 \end{equation*}
\end{proof}


\appendix

\section{Auxiliary material}\label{sec:aux}

In this section we collect some technical results useful for the proof of Theorem~\ref{thm:LS} and its applications. 

\begin{lemma}\label{lem:partial_der}
	Let $\nu_1,\dots,\nu_d$ be an orthonormal basis in $\RR^d$ and $U \subset \RR^d$ be open. Then,
	\begin{equation*}
	\sum_{\abs{\alpha} = m} \frac{1}{\alpha!} \partial^\alpha f \cdot \partial^\alpha g
	=
	\sum_{\abs{\alpha} = m} \frac{1}{\alpha!} \tilde{\partial}^\alpha f \cdot \tilde{\partial}^\alpha g
	\quad\text{ on }\quad
	U
	\end{equation*}
	for all $m \in \NN_0$ and all $f,g \in C^\infty(U)$, where
	$\tilde{\partial}^\alpha := \partial_{\nu_1}^{\alpha_1} \dots \partial_{\nu_d}^{\alpha_d}$.
	
	\begin{proof}
		We proceed by induction over $m$. The case $m = 0$ is clear, and the case $m=1$ follows from
		\begin{align*}
			\sum_{j=1}^d \partial_{\nu_j}f(x) \cdot \partial_{\nu_j}g(x)
			&=
			\sum_{j=1}^d \langle \nu_j,\nabla f \rangle_{\RR^d} \cdot \langle \nu_j,\nabla g \rangle_{\RR^d}
			=
			\langle \sum_{j = 1}^d \langle \nu_j, \nabla g \rangle_{\RR^d} \nu_j,\nabla f \rangle_{\RR^d}\\
			&=
			\langle \nabla f,\nabla g \rangle_{\RR^d}
			=
			\sum_{j=1}^d \partial_j f \cdot \partial_j g.
		\end{align*}	
		Suppose now that the claim is true for some $m \in \NN_0$. Then,
		\begin{equation*}
			\sum_{j=1}^d \partial_{\nu_j} \tilde{\partial}^\beta f \cdot \partial_{\nu_j} \tilde{\partial}^\beta g
			=
			\sum_{j=1}^d \partial_j \tilde{\partial}^\beta f \cdot \partial_j \tilde{\partial}^\beta g
		\end{equation*}
		by the case $m=1$ applied to $\tilde{\partial}^\beta f$ and $\tilde{\partial}^\beta g$. Hence,
		\begin{align*}
			\sum_{\abs{\alpha} = m+1} \frac{1}{\alpha!} \tilde{\partial}^\alpha f \cdot \tilde{\partial}^\alpha g
			&=
			\frac{1}{m+1} \sum_{\abs{\beta} = m} \frac{1}{\beta!} \sum_{j = 1}^d \partial_{\nu_j}\tilde{\partial}^\beta f \cdot
			\partial_{\nu_j}\tilde{\partial}^\beta g\\
			&=
			\frac{1}{m+1} \sum_{j = 1}^d \sum_{\abs{\beta} = m} \frac{1}{\beta!} \tilde{\partial}^\beta \partial_j f \cdot
			\tilde{\partial}^\beta \partial_j g\\
			&=
			\frac{1}{m+1} \sum_{j = 1}^d \sum_{\abs{\beta} = m} \frac{1}{\beta!} \partial^\beta \partial_j f \cdot
			\partial^\beta \partial_j g\\
			&=
			\sum_{\abs{\alpha} = m+1} \frac{1}{\alpha!} \partial^\alpha f \cdot \partial g.\qedhere
		\end{align*}
	\end{proof}%
\end{lemma}

For the next result, we identify, as in Section~\ref{subsec:transformations}, a matrix in $\RR^d$ with the corresponding induced
linear mapping on $\RR^d$ with respect to the standard basis.

\begin{lemma}\label{lemma:derivative-transformation}
	Let $f,g \in W^{\infty, 2}(\Omega)$ with some open $\Omega \subset \RR^d$.
	\begin{enumerate}[(a)]
		
		\item
		If $U \in \RR^{d\times d}$ is orthogonal, then $f \circ U, g \circ U \in W^{\infty, 2}(U^T(\Omega))$ with
		\begin{equation*}
		\sum_{\abs{\alpha}=m} \frac{1}{\alpha!} \langle \partial^\alpha (f\circ U) , \partial^\alpha (g \circ U) \rangle_
		{L^2(U^T(\Omega))}
		=
		\sum_{\abs{\alpha}=m} \frac{1}{\alpha!} \langle \partial^\alpha f , \partial^\alpha g \rangle_{L^2(\Omega)}
		\quad \forall\, m\in\NN_0
		.
		\end{equation*}
		
		\item
		If $P \in \RR^{d\times d}$ is symmetric and positive definite, then $f\circ P \in W^{\infty, 2}(P^{-1}(\Omega))$ with
		\begin{equation*}
		\sum_{\abs{\alpha}=m} \frac{1}{\alpha!}\norm{\partial^\alpha (f\circ P)}{L^2(P^{-1}(\Omega))}^2
		\geq 
		\frac{p_{\min}^{2m}}{\det P}\sum_{\abs{\alpha}=m} \frac{1}{\alpha!}\norm{\partial^\alpha f}{L^2(\Omega)}^2
		\quad \forall\, m\in\NN_0,
		\end{equation*}
		where $p_{\min}$ is the smallest eigenvalue of $P$.
		
	\end{enumerate}
\end{lemma}

\begin{proof}
	(a).~%
	Let $e_1, \ldots, e_d$ denote the standard basis in $\RR^d$, and set $\nu_j:= U^{-1} e_j$ for $j=1, \ldots d$. Then,
	$\nabla(f\circ U) (x) = U^T(\nabla f)(Ux)$ and, consequently,
	\[
	\partial_{\nu_j}(f\circ U)(x)
	=
	\langle U^T(\nabla f) (Ux), \nu_j \rangle_{\CC^d}
	=
	\langle (\nabla f) (Ux), U\nu_j \rangle_{\CC^d}
	=
	(\partial_j f)(Ux)
	.
	\]
	Iterating this identity gives
	$\partial_{\nu_1}^{\alpha_1} \ldots \partial_{\nu_d}^{\alpha_d}(f\circ U)(x) = (\partial^\alpha f) (Ux)$ for all
	$\alpha \in \NN_0^d$; the same holds, of course, also for $g$ instead of $f$. Abbreviating
	$\tilde{\partial}^\alpha := \partial_{\nu_1}^{\alpha_1} \ldots \partial_{\nu_d}^{\alpha_d}$, for $m\in\NN_0$ we now obtain with
	the use of Lemma~\ref{lem:partial_der} that
	\begin{multline*}
	\sum_{\abs{\alpha}=m} \frac{1}{\alpha!} \langle \partial^\alpha (f\circ U) , \partial^\alpha (g \circ U) \rangle_
	{L^2(U^T(\Omega))}
	=
	\sum_{\abs{\alpha}=m} \frac{1}{\alpha!} \langle \tilde{\partial}^\alpha (f\circ U) , \tilde{\partial}^\alpha (g\circ U)
	\rangle_{L^2(U^T(\Omega))}\\
	=
	\sum_{\abs{\alpha}=m} \frac{1}{\alpha!} \langle (\partial^\alpha f)\circ U , (\partial^\alpha g)\circ U \rangle_
	{L^2(U^T(\Omega))}
	=
	\sum_{\abs{\alpha}=m} \frac{1}{\alpha!} \langle \partial^\alpha f , \partial^\alpha g \rangle_{L^2(\Omega)}
	.
	\end{multline*}
	
	(b).~%
	Let $U\in \RR^{d\times d}$ be orthogonal such that $U^T P U = D := \diag (p)$, where	$p$ is the vector containing all
	eigenvalues	of $P$. Using part~(a) twice, for $m\in\NN_0$ we now obtain that
	\begin{align*}
	\sum_{\abs{\alpha}=m} \frac{1}{\alpha!}\norm{\partial^\alpha (f\circ P)}{L^2(P^{-1}(\Omega))}^2
	&= 
	\sum_{\abs{\alpha}=m} \frac{1}{\alpha!}\norm{\partial^\alpha (f\circ (UD))}{L^2((D^{-1}U^T)(\Omega))}^2\\
	&\geq
	\frac{p_{\min}^{2m}}{\det D}\sum_{\abs{\alpha}=m} \frac{1}{\alpha!}\norm{\partial^\alpha (f\circ U)}{L^2(U^T(\Omega))}^2\\
	&=
	\frac{p_{\min}^{2m}}{\det P} \sum_{\abs{\alpha}=m} \frac{1}{\alpha!}\norm{\partial^\alpha f}{L^2(\Omega)}^2
	.\qedhere
	\end{align*}
\end{proof}%

\section{The \texorpdfstring{$d$}{d}-dimensional harmonic oscillator}
This appendix briefly revisits the $d$-dimensional harmonic oscillator and reproduces a corresponding Bernstein-type inequality
for finite linear combinations of Hermite functions from~\cite[Proposition~4.3\,(ii)]{BJPS18}, adapted to our framework.

For $\beta = (\beta_1,\dots,\beta_d) \in \NN_0^d$ let $\Phi_\beta \colon \RR^d \to \RR$ be given by
\begin{equation*}
	\Phi_\beta
	:=
	\prod_{j=1}^d \phi_{\beta_j}(x_j)
	,\quad
	x = (x_1,\dots,x_d) \in \RR^d
	,
\end{equation*}
where $\phi_k \colon \RR \to \RR$, $k \geq 0$, denotes the $k$-th standard Hermite function,
\begin{equation*}
	\phi_k(t)
	=
	\frac{(-1)^k}{\sqrt{2^k k! \sqrt{\pi}}} \ee^{t^2/2} \frac{\mathrm{d}^k}{\dd t^k}\ee^{-t^2}
	.
\end{equation*}
The family $(\Phi_\beta)_{\beta\in\NN_0^d}$ gives rise to an orthonormal basis of $L^2(\RR^d)$ consisting of eigenfunctions of
the $d$-dimensional harmonic oscillator $-\Delta_{\RR^d} + \abs{x}^2$ with
\begin{equation}\label{eq:harmonic:eigenvalues}
	(-\Delta_{\RR^d} + \abs{x}^2) \Phi_\beta
	=
	(2\abs{\beta}+d)\Phi_\beta
	.
\end{equation}
In particular, $-\Delta_{\RR^d}+\abs{x}^2$ has no spectrum below $d$.

\begin{prop}[{see~\cite[Proposition~4.3\,(ii)]{BJPS18}}]\label{prop:BJPS}
	For all $\lambda \ge d$, $\Ran\EE_{-\Delta_{\RR^d}+\abs{x}^2}(\lambda)$ is contained in $W^{\infty,2}(\RR^d)$. Moreover, given
	$\delta > 0$, every $f \in \Ran\EE_{-\Delta_{\RR^d}+\abs{x}^2}(\lambda)$ satisfies
	\begin{equation*}
		\sum_{\abs{\alpha}=m} \frac{1}{\alpha!}\norm{\partial^\alpha f}{L^2(\RR^2)}^2
		\leq 
		\frac{C_B(m,\lambda)}{m!} \norm{f}{L^2(\RR^d)}^2
		\quad\text{ for all }\
		m \in \NN_0
	\end{equation*}
	with
	\begin{equation*}
		C_B(m,\lambda)
		=
		(2\delta)^{2m} \ee^{\ee/\delta^2} (m!)^2 \ee^{2\sqrt{\lambda}/\delta}
		.
	\end{equation*}
\end{prop}

\begin{proof}
	By~\eqref{eq:harmonic:eigenvalues}, we have
	$\Ran\EE_{-\Delta_{\RR^d}+\abs{x}^2}(\lambda) = \Span_\CC\{ \Phi_\beta \colon \abs{\beta} \leq N \} \subset W^{\infty,2}(\RR^d)$
	with $N \in \NN_0$ such that $2N+d \le \lambda < 2N+d+2$. Let $f \in \Ran\EE_{-\Delta_{\RR^d}+\abs{x}^2}(\lambda)$. Then, $f$
	belongs to the Schwartz space, and an integration by parts (cf.~Lemma~\ref{lem:Bernstein-Criterion} and
	Remark~\ref{rmk:Bernstein}) yields
	\begin{align*}
		\sum_{j=1}^d \norm{\partial_j f}{L^2(\RR^d)}^2
		&=
		\langle f , (-\Delta_{\RR^d}) f \rangle_{L^2(\RR^d)}\\
		&\leq
		\langle f , (-\Delta_{\RR^d} + \abs{x}^2) f \rangle_{L^2(\RR^d)}
		\leq
		(2N+d)\norm{f}{L^2(\RR^d)}^2
		.
	\end{align*}
	Since also $\partial^\alpha f \in \Span_\CC\{ \Phi_\beta \colon \abs{\beta} \leq N + \abs{\alpha} \}$ for all
	$\alpha \in \NN_0^d$, as shown in~\cite{BJPS18}, a straightforward induction (cf.~the proof of Lemma~\ref{lem:partial_der})
	gives
	\begin{equation*}
		\sum_{\abs{\alpha}=m} \frac{1}{\alpha!} \norm{\partial^\alpha f}{L^2(\RR^d)}^2
		\leq
		\frac{1}{m!} \norm{f}{L^2(\RR^d)}^2 \prod_{k=0}^{m-1} (2N + d + 2k)
		\leq
		\frac{1}{m!} \norm{f}{L^2(\RR^d)}^2 \prod_{k=0}^{m-1} (\lambda + 2k)
	\end{equation*}
	for all $m \in \NN_0$; cf.~\cite[(4.38)]{BJPS18}.
	
	It remains to estimate the product in the right-hand side of the last inequality.	To this end, we proceed similarly as in the
	proof of~\cite[Proposition~4.3\,(ii)]{BJPS18}: If $2m-2 \le \lambda$, then $\lambda + 2k \leq 2\lambda$, so that
	\begin{equation*}
		\prod_{k=0}^{m-1} (\lambda + 2k)
		\leq
		(2\lambda)^m
		=
		(\sqrt{2}\delta)^{2m} \biggl( \frac{\sqrt{\lambda}}{\delta} \biggr)^{2m}
		\leq
		(\sqrt{2}\delta)^{2m} \ee^{2\sqrt{\lambda}/\delta} (m!)^2
		\leq
		C_B(m,\lambda)
		,
	\end{equation*}
	where we have used the elementary inequality $t^m \leq m! \ee^t$ for $t \geq 0$.
	
	If, on the other hand, $\lambda \leq 2m - 2$, then $\lambda + 2k \le 4m$, and we obtain in a similar way
	\begin{equation*}
		\prod_{k=0}^{m-1} (\lambda + 2k)
		\leq
		(4m)^m
		\leq
		(2\delta)^{2m} m! \Bigl( \frac{\ee}{\delta^2} \Bigr)^m
		\leq
		(2\delta)^{2m} (m!)^2 \ee^{\ee/\delta^2}
		\leq
		C_B(m,\lambda)
		,
	\end{equation*}
	which completes the proof.
\end{proof}%

\section{Choosing suitable coverings and proof of Lemma~\ref{lem:covering}}\label{sec:cov}

We here discuss how to construct suitable coverings of various types of domains in $\RR^d$. The parameters of the respective
coverings shall not depend on the scale of the domain as long as the latter contains an open hypercube with sides of length
$\varrho$ parallel to coordinate axes. Recall that we are looking only for essential coverings in the sense that we only need to
cover the domain up to a set of measure zero. In some cases, this allows to avoid any overlap between the covering elements.
However, at the cost of maybe some overlap, we give priority to small side length $l$ of the hyperrectangles
surrounding the covering elements. With regard to Theorem~\ref{thm:LS}, this is a sound strategy because the overlap $\kappa$
enters the final estimate only logarithmically in the exponent.

We begin with the following elementary observation that allows us to combine coverings of two domains to a corresponding covering
of their Cartesian product.

\begin{lemma}\label{lem:Cartesian}
	If $\{Q_j\}_{j\in J}$ is a $(\kappa,\varrho,l,\eta)$-covering of a domain $\Omega \subset \RR^d$ and $\{Q_{j'}'\}_{j'\in J'}$ is
	a	$(\kappa',\varrho,l',\eta')$-covering of a domain $\Omega' \subset \RR^{d'}$, then the family
	$\{ Q_j \times Q_{j'}'\}_{(j,j')\in J\times J'}$ yields a $(\kappa\kappa',\varrho,(l,l'),\tilde{\eta})$-covering of
	$\Omega \times \Omega' \subset \RR^{d+d'}$ with
	\begin{equation*}
	\tilde{\eta}
	=
	\eta\eta'\,\frac{(d')^{d'/2} \cdot d^{d/2}}{(d+d')^{(d+d')/2}}
	.
	\end{equation*}
	In particular, if $\eta \ge c/d^{d/2}$ and $\eta' \ge c'/(d')^{d'/2}$, then $\tilde{\eta} \ge cc'/((d+d')^{(d+d')/2})$.
\end{lemma}

\begin{proof}
	Only the parameter $\tilde{\eta}$ needs explicit treatment. Here, we consider the linear bijection
	$\bijection_{j,j'} \colon \RR^{d} \times \RR^{d'} \to \RR^{d} \times \RR^{d'}$ with
	\begin{equation*} 
	\bijection_{j,j'}(x,y)
	=
	\bigl( r\bijection_j(x)/\diam(\bijection_j(Q_j)),\bijection'_{j'}(y)/\diam(\bijection'_{j'}(Q'_{j'})) \bigr)
	,\quad
	r^2 = d/d',
	\end{equation*}
	where $\bijection_j$ and $\bijection'_{j'}$ are the bijections corresponding to $Q_j$ and $Q'_{j'}$, respectively. Then,
	$\abs{\bijection_{j,j'}(Q_j \times Q'_{j'})} \ge \eta\eta'r^d$ and $\diam(\bijection_{j,j'}(Q_j \times Q'_{j'}))^2 = 1 + r^2$,
	with which it is easily verified that
	\begin{equation*}
	\abs{\bijection_{j,j'}(Q_j \times Q'_{j'})}
	\geq
	\tilde{\eta} \diam(\bijection_{j,j'}(Q_j \times Q'_{j'}))^{d+d'}
	.\qedhere
	\end{equation*}
\end{proof}%

\subsection{Generalised rectangles}\label{subsec:genRec}

The above lemma allows to reduce the construction of coverings of generalised rectangles $\Omega = \bigtimes_{j=1}^d (a_j,b_j)$
with $\varrho \le b_j - a_j$ for all $j$ to the one-dimensional case, that is, open intervals $\Omega \subset \RR$. Here, we may
just cover $\Omega$ with open intervals of length $\varrho$. This can be done in an adjacent manner without overlap if $\Omega$
is unbounded and with an overlap of at most two intervals otherwise. In any case, we end up with a
$(2,\varrho,\varrho,1)$-covering of $\Omega \subset \RR$.

Combining the above one-dimensional consideration with Lemma~\ref{lem:Cartesian}, we conclude
that given $\varrho>0$, every generalised rectangle with $\varrho\geq b_j-a_j$ for all $j$ possesses a 
$(2^k,\varrho,l,d^{-d/2})$-covering with $l = (\varrho,\dots,\varrho) \in (0,\infty)^d$ and $k$ the
number of its bounded coordinates. In particular, this covering is a $(2^d,\varrho,l,d^{-d/2})$-covering.

\subsection{Sectors}\label{subsec:sectors}

For $0 < \theta < \pi/2$, let $\Omega = S_\theta \subset \RR^2$ denote the sector
\begin{equation*}
S_\theta = \bigl\{ (x_1,x_2) \in (0,\infty)^2 \mid 0 < x_2 < x_1\tan\theta \bigr\}.
\end{equation*}
In order to deal with the cusp of this sector, we consider the covering element
$Q:= S_\theta \cap \big((0,q)\times (0,\varrho)\big)$ with $q=\varrho(1+\cot\theta)$, which is just large enough to contain a
square of side length $\varrho$. Moreover, choosing the linear bijection $\bijection \colon \RR^2 \to \RR^2$ to map the rectangle
$(0,q)\times(0,\varrho)$ onto the square $(0,\varrho)^2$, we easily see that
\begin{equation*}
	\frac{\abs{\bijection(Q)}}{\diam(\bijection(Q))^2}
	\geq
	\frac{1}{4}
	.
\end{equation*}
We now cover $S_\theta$ without overlap with squares of side length $\varrho$ and translates of $Q$ as depicted in
Figure~\ref{fig:sectorsandtriangles}, which results in a $(1,\varrho,(q,\varrho),1/4)$-covering of $S_\theta$ with
$q = \varrho(1+\cot\theta)$. Here, it is worth to note that $q$ gets large as the opening angle $\theta$ gets small.

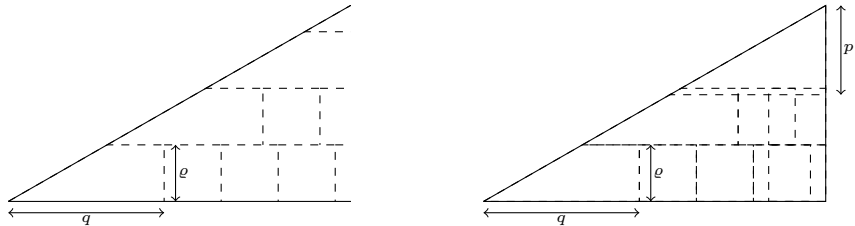
\begin{figure}[ht]
	\begin{tikzpicture}[scale=0.5]
	\pgfmathsetmacro{\a}{30}; 
	\pgfmathsetmacro{\t}{tan(\a)}; 
	\pgfmathsetmacro{\r}{1.5} 
	\begin{scope}[xshift=-12.5cm]
		\clip (0,-1) rectangle (9, 9*\t);
	\draw (0,0) -- (11,0); 
	\draw (0,0) -- (11, 11*\t);
	\foreach \y in {0,\r, 2*\r, 3*\r}{
		\draw[dashed] (\y/\t,\y) -- (\r/\t+\y/\t, \r+\y) -- (\r+\r/\t+\y/\t,\r+\y) -- (\r+\r/\t+\y/\t,0+\y);
		\foreach \x in {0,\r, 2*\r, 3*\r}{
			\draw[dashed] (\x+\r+\r/\t+\y/\t, \r+\y) -- (\x+2*\r+\r/\t+\y/\t,\y+\r) --(\x+2*\r+\r/\t+\y/\t, \y);
		}
	}
	\draw[<->] (0,-0.3) -- (\r+\r/\t, -0.3); \node at (0.5*\r+ 0.5*\r/\t, -0.5) {\tiny $q$};
	\draw[<->] (\r+\r/\t+0.3, 0) -- (\r+\r/\t+0.3, \r); \node at (\r+\r/\t+0.5, \r/2) {\tiny$\varrho$};
	\end{scope}

	\begin{scope}
	\draw (0,0) -- (9,0) -- (9,9*\t) -- (0,0);
	\draw[dashed] (9, 9*\t) -- (9, 9*\t -\r - \r*\t) -- (9-\r-\r/\t,9*\t -\r -\r*\t) -- (9, 9*\t);
	\foreach \y in {0,\r}{
		\draw[dashed] (\y/\t,\y) -- (\r/\t+\y/\t, \r+\y) -- (\r+\r/\t+\y/\t,\r+\y) -- (\r+\r/\t+\y/\t,0+\y) -- (\y/\t,\y);
	}
	\draw[dashed] (\r+\r/\t,\r) rectangle (\r+\r+\r/\t,0); 
	\draw[dashed] (\r+\r+\r/\t,\r) rectangle (\r+\r+\r+\r/\t,0);
	\draw[dashed] (\r+\r+\r+\r/\t,\r) rectangle (\r+\r+\r+\r+\r/\t,0);
	\draw[dashed] (9-\r, \r) rectangle (9,0);
	
	\draw[dashed] (9-\r, \r+\r) rectangle (9,\r);
	\draw[dashed] (\r+\r/\t+\r/\t,\r+\r) rectangle (\r+\r+\r/\t+\r/\t,\r);
	 
	\draw[<->] (9.4, 9*\t) -- (9.4, 9*\t -\r - \r*\t); \node at (9.6, 7*\t) {\tiny $p$};
	\draw[<->] (0,-0.3) -- (\r+\r/\t, -0.3); \node at (0.5*\r+ 0.5*\r/\t, -0.5) {\tiny $q$};
	\draw[<->] (\r+\r/\t+0.3, 0) -- (\r+\r/\t+0.3, \r); \node at (\r+\r/\t+0.5, \r/2) {\tiny$\varrho$};
	\end{scope}
	\end{tikzpicture}
	\caption{On the left a covering of the sector $S_\theta$, on the right an adapted one of the triangle $\cR_{\theta,L}$ with
	$L \ge q = \varrho(1+\cot\theta)$.}
	\label{fig:sectorsandtriangles}
\end{figure}

\subsection{Right-angled triangles}\label{subsec:righttriangles}
For $0 < \theta < \pi/2$ and $L>0$ consider the right-angled triangle $\cR_{\theta,L} := S_\theta \cap ((0,L) \times (0,\infty))$. 
The considerations for the sector $S_\theta$ above show that we have to require $L \ge q = \varrho(1+\cot\theta)$ in order to
ensure that $\cR_{\theta,L}$ contains a square with sides of length $\varrho$ parallel to coordinate axes. In addition to such
squares and the element $Q$ discussed for the sector, we now cover $\cR_{\theta,L}$ also with a translate of
$\cR := \cR_{\theta,q}= S_\theta\cap ((0,q) \times (0,\infty))$, which is half of the rectangle $(0,q) \times (0,p)$ with
$p := \varrho(1+\tan\theta)$, in order to deal with the upper cusp of $\cR_{\theta,L}$. The linear bijection
$\bijection \colon \RR^2 \to \RR^2$ mapping the rectangle $(0,q) \times (0,p)$ to the square $(0,1)^2$ clearly satisfies
\begin{equation*}
	\frac{\abs{\bijection(\cR)}}{\diam(\bijection(\cR))^2}
	=
	\frac{1}{4}
	.
\end{equation*}

Upon covering $\cR_{\theta,L}$ with the above elements as depicted in Figure~\ref{fig:sectorsandtriangles}, we see that no more
than two of those squares and/or the translate of $\cR$ participate in an overlap, so that we have a total overlap of at most
three of those elements. Hence, for $L \geq \varrho(1+\cot\theta)$, this leads to a (finite) $(3,\varrho,(q,p),1/4)$-covering of
$\cR_{\theta,L}$ with $q = \varrho(1+\cot\theta)$ and $p = \varrho(1+\tan\theta)$.

\subsection{Equilateral triangles}\label{subsec:equilateral}

For $L > 0$ consider the equilateral triangle $\cT_L \subset \RR^2$ with vertices $(L/2,-\sqrt{3}L/6)$, $(0,\sqrt{3}L/3)$, and
$(-L/2,-\sqrt{3}L/6)$. It has sides of length $L$ and the origin as centre of mass.
Clearly, $\cT_L$ contains a square with sides of length $\varrho$ parallel to coordinate axes if and only if
$L \ge \sqrt{3}\varrho$. Denote $\cT := \cT_{\sqrt{3}\varrho}$, and let $\cT'$ denote the rotation of $\cT$ 
by angle $\pi$ around the origin. 
Both $\cT$ and $\cT'$ lie in a rectangle with sides of length $l=(\sqrt{3}\varrho,3\varrho/2)$ 
parallel to
coordinate axes. Clearly,
\begin{equation*}
\frac{\abs{\cT}}{\diam(\cT)^2}
=
\frac{\abs{\cT'}}{\diam(\cT')^2}
=
\frac{\sqrt{3}}{4}
.
\end{equation*}

We now cover $\cT_L$ with translates of $\cT$ and $\cT'$ as depicted in Figure~\ref{fig:equilateral}.
Here, a large portion of $\cT_L$, namely a translate of $\cT_{\sqrt{3}k\varrho}$ with
$\sqrt{3}k\varrho \le L < \sqrt{3}(k+1)\varrho$, $k \in \NN$, can be covered this way without any overlap. For the remaining part
of $\cT_L$, we can guarantee that no more than three different elements participate in an overlap, so that for
$L \geq \sqrt{3}\varrho$ we end up with a (finite) $(3,\varrho,l,\sqrt{3}/4)$-covering of $\cT_L$ with
$l = (\sqrt{3}\varrho,3\varrho/2)$.

\begin{figure}[ht]
	\begin{tikzpicture}[scale=0.5]
	\begin{scope}
	\pgfmathsetmacro{\a}{30};
	\pgfmathsetmacro{\t}{tan(30)};
	\pgfmathsetmacro{\ct}{1/tan(30)};
	\pgfmathsetmacro{\r}{2};
	\pgfmathsetmacro{\L}{3.5*\ct*\r};
	\pgfmathsetmacro{\R}{\L/(\ct*\r)-1.5};
	\draw (-\L/2, -\ct*\L/6) --(\L/2, -\ct*\L/6) --(0,\ct*\L/3) -- (-\L/2, -\ct*\L/6); 
	
	\foreach \y in {0,...,\R}{
		\pgfmathsetmacro{\P}{\R-\y};
		\foreach \x in {0,...,\P}{
		\draw[dashed] (-\L/2+\x*\ct*\r+\y*\ct*\r/2, -\ct*\L/6+\ct*\y*\ct*\r/2) -- (-\L/2+\ct*\r+\x*\ct*\r+\y*\ct*\r/2, -\ct*\L/6+\ct*\y*\ct*\r/2) -- (-\L/2+\ct*\r+\x*\ct*\r-\ct*\r/2+\y*\ct*\r/2, -\ct*\L/6+\ct*\ct*\r/2+\ct*\y*\ct*\r/2) -- 
		(-\L/2+\x*\ct*\r+\y*\ct*\r/2, -\ct*\L/6+\ct*\y*\ct*\r/2);
		}
	}
	\draw[dashed] (\L/2, -\ct*\L/6) -- (\L/2-\ct*\r, -\ct*\L/6) -- (\L/2-\ct*\r/2, -\ct*\L/6+\ct*\ct*\r/2) -- (\L/2, -\ct*\L/6);
	\draw[dashed] (\L/2-\ct*\r/2, -\ct*\L/6+\ct*\ct*\r/2) -- (\L/2-\ct*\r/2-\ct*\r, -\ct*\L/6+\ct*\ct*\r/2) -- (\L/2-\ct*\r, -\ct*\L/6);
	
	\draw[dashed] (\L/2-\ct*\r/2, -\ct*\L/6+\ct*\ct*\r/2) -- (\L/2-\ct*\r/2-\ct*\r/2, -\ct*\L/6+\ct*\ct*\r/2+\ct*\ct*\r/2) -- (\L/2-\ct*\r/2-\ct*\r, -\ct*\L/6+\ct*\ct*\r/2);
	\draw[dashed] (\L/2-\ct*\r/2-\ct*\r/2, -\ct*\L/6+\ct*\ct*\r/2+\ct*\ct*\r/2) -- (\L/2-\ct*\r/2-\ct*\r/2-\ct*\r, -\ct*\L/6+\ct*\ct*\r/2+\ct*\ct*\r/2) -- (\L/2-\ct*\r/2-\ct*\r, -\ct*\L/6+\ct*\ct*\r/2);
	\draw[dashed]  (\L/2-\ct*\r/2-\ct*\r/2, -\ct*\L/6+\ct*\ct*\r/2+\ct*\ct*\r/2) -- (\L/2-\ct*\r/2-\ct*\r/2-\ct*\r, -\ct*\L/6+\ct*\ct*\r/2+\ct*\ct*\r/2) -- (\L/2-\ct*\r/2-\ct*\r, -\ct*\L/6+3*\ct*\ct*\r/2) -- (\L/2-\ct*\r/2-\ct*\r/2, -\ct*\L/6+\ct*\ct*\r/2+\ct*\ct*\r/2);
	\draw[dashed] (0,\ct*\L/3) -- (\ct*\r/2, \ct*\L/3-\ct*\ct*\r/2) -- (\ct*\r/2-\ct*\r, \ct*\L/3-\ct*\ct*\r/2) -- (0,\ct*\L/3);
	\draw[<->] (-\L/2,-\ct*\L/6-0.3) -- (\L/2, -\ct*\L/6-0.3); \node at (0, -\ct*\L/6-0.7) {\tiny $L$};
	\draw[<->] (-\L/2-0.25,-\ct*\L/6+0.15) -- (-\L/2+\ct*\r/2-0.25, -\ct*\L/6+\ct*\ct*\r/2+0.15); \node at (-\L/2+\ct*\r/4-0.9, -\ct*\L/6+\ct*\ct*\r/4+0.25) {\tiny $\sqrt{3}\varrho$};
	\end{scope}
	\end{tikzpicture}
	\caption{A covering of the	equilateral triangle $\cT_L$ with $L \geq 3\sqrt{3}\varrho$.}\label{fig:equilateral}
\end{figure}
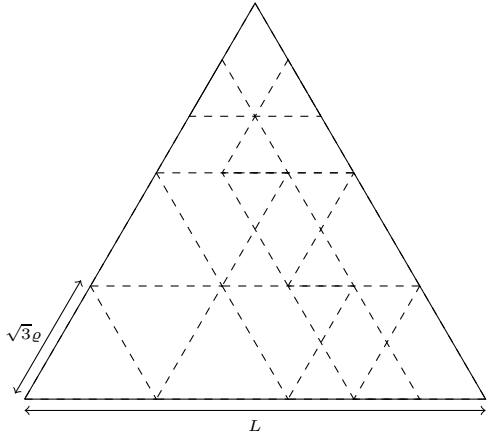

\subsection{Proof of Lemma~\ref{lem:covering}}

The statement of Lemma~\ref{lem:covering} is clearly correct for generalised rectangles by Subsection~\ref{subsec:genRec}, and
for the two dimensional domains in question we have above constructed coverings that are, in particular,
$(4,\varrho,l,1/4)$-coverings with the respective parameters $l \in [\varrho,\infty)^2$. Here, the isosceles right-angled
triangles are represented by $\cR_{\theta,L}$ with $\theta = \pi/4$, whereas the hemiequilateral triangles correspond to
$\cR_{\theta,L}$ with $\theta = \pi/3$. This means that in each case the entries of $l$ can be bounded by multiples of $\varrho$,
which depend at most on the opening angle $\theta$ if a sector is considered (the smaller $\theta$, the larger the multiple).
Hence, the claim of Lemma~\ref{lem:covering} holds also for the stated two dimensional domains. Since there are at most $d/2$
factors of dimension $2$ in a Cartesian product in $\RR^d$, the claim for Cartesian products of the above domains now follows
from Lemma~\ref{lem:Cartesian}. This completes the proof of Lemma~\ref{lem:covering}.

\section*{Acknowledgements}
The authors are indebted to Ivan Veseli\'c for introducing them to this field of research. They also thank the anonymous reviewer
for helpful remarks on the manuscript.

\end{document}